\documentclass[12pt]{article}
\usepackage{amssymb}
\usepackage{amsmath}
\usepackage{amsthm}
\usepackage{color}
\usepackage{comment}
\usepackage{enumitem}
\usepackage{graphicx}
\setitemize{itemsep=0pt}

\DeclareFontFamily{OT1}{pzc}{}
\DeclareFontShape{OT1}{pzc}{m}{it}{<-> s * [1.10] pzcmi7t}{}
\DeclareMathAlphabet{\mathpzc}{OT1}{pzc}{m}{it}

\begin{document}

\def\b0{{\bf 0}}
\def\bx{{\bf x}}
\def\by{{\bf y}}
\def\bz{{\bf z}}
\def\ee{\varepsilon}
\def\cM{\mathcal{M}}
\def\co{\mathpzc{o}}
\def\cO{\mathcal{O}}

\newcommand{\removableFootnote}[1]{}

\newtheorem{theorem}{Theorem}[section]
\newtheorem{conjecture}[theorem]{Conjecture}
\newtheorem{corollary}[theorem]{Corollary}
\newtheorem{lemma}[theorem]{Lemma}

\theoremstyle{definition}
\newtheorem{definition}{Definition}[section]

\title{
The stability of fixed points on switching manifolds of piecewise-smooth continuous maps.
}
\author{
D.J.W.~Simpson\\\\
Institute of Fundamental Sciences\\
Massey University\\
Palmerston North\\
New Zealand
}
\maketitle

\begin{abstract}
This paper concerns piecewise-smooth maps on $\mathbb{R}^d$ that are continuous but not differentiable on switching manifolds (where the functional form of the map changes). The stability of fixed points on switching manifolds is investigated in scenarios for which one-sided derivatives are locally bounded. The lack of differentiability allows fixed points to be Milnor attractors despite being unstable. For this reason a measure-theoretic notion of stability is considered in addition to standard notions of stability. Locally the map is well approximated by a piecewise-linear map that is linearly homogeneous when the fixed point is at the origin. For the class of continuous, linearly homogeneous maps, and $\co(\bx)$ perturbations of these maps, a sufficient condition for the exponential stability of the origin is obtained. It is shown how the stability of the origin can be determined by analysing invariant probability measures of a map on $\mathbb{S}^{d-1}$. The results are illustrated for the two-dimensional border-collision normal form. The fixed point may be asymptotically stable even if both smooth components of the map are area-expanding, and unstable even if it is the $\omega$-limit set of almost all points in $\mathbb{R}^d$.
\end{abstract}

\section{Introduction}
\label{sec:intro}

Piecewise-smooth dynamical systems provide useful mathematical models for a wide variety of physical and abstract phenomena
involving a threshold, switch, or other type of abrupt event.
Classical applications include switched control systems \cite{Jo03,TaLa12,Ut92},
vibro-impacting mechanical systems \cite{AwLa03,Br99,WiDe00},
and systems with stick-slip friction \cite{Ab00,BlCz99,OeHi96}.
Piecewise-smooth maps arise as return maps of piecewise-smooth systems of differential equations and
as discrete-time mathematical models, particularly in economics \cite{PuSu06}.

As the parameters of a piecewise-smooth continuous map are varied, a border-collision bifurcation occurs
when a fixed point collides with a switching manifold
and locally the map is piecewise-linear to leading order \cite{DiBu08}.
Border-collision bifurcations can be the cause of complicated dynamics including chaos
and this has been described in diverse areas of application, see for instance \cite{An03,SzOs09,Ti02,ZhMo06b}.
The dynamics near a non-degenerate border-collision bifurcation is determined by the eigenvalues
of the two matrices in the piecewise-linear form,
yet explicit criteria for the existence of invariant sets in terms of these eigenvalues
is available only for fixed points and period-two solutions \cite{Fe78,Si14d}.
Criteria for the existence of higher period solutions and some other relatively simple features can be given
implicitly and then solved to a high degree of accuracy with a root finding method,
but in general it is usually difficult to ascertain the nature of the dynamics near a border-collision bifurcation
prior to performing detailed calculations \cite{Si16}.

A basic, important, and unresolved question (even in two dimensions) is:
how can we tell from the two sets of eigenvalues in the piecewise-linear form
whether or not there exists a local attractor?
Intuitively, knowledge of the stability of the fixed point at the bifurcation
should go a long way to answering this question.
Indeed for ODEs, the stability of the equilibrium at a Hopf bifurcation
determines the criticality of the bifurcation \cite{MaMc76},
and for the continuous-time analog of a border-collision bifurcation,
asymptotic stability of the equilibrium at the bifurcation implies
the existence of a persistent attracting set \cite{DiNo08}.

This paper is concerned with the stability of fixed points at border-collision bifurcations.
Consider a piecewise-$C^1$ continuous map on $\mathbb{R}^d$
and suppose that the origin, denoted $\b0$, is a fixed point on a switching manifold.
Locally, a curved switching manifold can be straightened via a series of near-identity coordinate changes \cite{DiBu01},
and so for simplicity it is assumed that the switching manifold is linear.
With these assumptions, in a neighbourhood of $\b0$ the map can be written as
\begin{equation}
f(\bx) =
\begin{cases}
A_L \bx + \co(\bx) \;, & B^{\sf T} \bx \le 0 \;, \\
A_R \bx + \co(\bx) \;, & B^{\sf T} \bx \ge 0 \;,
\end{cases}
\label{eq:f}
\end{equation}
where $B \in \mathbb{R}^d \setminus\! \{ \b0 \}$ and $A_L$ and $A_R$ are real-valued $d \times d$ matrices.
The symbol $\co$ represents little-o notation.
That is, $\frac{\| \co(\bx) \|}{\| \bx \|} \to 0$ as $\bx \to \b0$.
The assumption of continuity on $B^{\sf T} \bx = 0$
implies $A_L - A_R = C B^{\sf T}$, for some $C \in \mathbb{R}^d$\removableFootnote{
I originally thought to motivate only assuming differentiability by referring to
grazing-sliding bifurcations, unfortunately don't think any of the results
in this paper can be applied to grazing-sliding bifurcations because
they have $g(\bx) = \b0$ for some $\bx \ne \b0$.
}.

The dynamics of \eqref{eq:f} near $\b0$ is well-approximated by
\begin{equation}
g(\bx) =
\begin{cases}
A_L \bx \;, & B^{\sf T} \bx \le 0 \;, \\
A_R \bx \;, & B^{\sf T} \bx \ge 0 \;.
\end{cases}
\label{eq:g}
\end{equation}
It is already known that the stability of $\b0$ for \eqref{eq:g} is not elementary.
For instance $\b0$ can be unstable in cases for which all eigenvalues of $A_L$ and $A_R$ have modulus less than one
giving rise to so-called dangerous border-collision bifurcations \cite{DoBa06,GaBa05,HaAb04}.

The map \eqref{eq:g} is {\em linearly homogeneous} in the sense that
\begin{equation}
g(\alpha \bx) = \alpha g(\bx) \;, \quad {\rm for~all~} \bx \in \mathbb{R}^d {\rm ~and~all~} \alpha \ge 0 \;.
\label{eq:linearlyHomogeneous}
\end{equation}
A linearly homogeneous map is determined by 
its values on the $(d-1)$-dimensional unit sphere,
$\mathbb{S}^{d-1} = \{ \bx \in \mathbb{R}^d ~|~ \| \bx \| = 1 \}$,
and can be written using two lower-dimensional functions as follows.
Assuming $g$ is continuous and $g(\bx) = \b0$ only for $\bx = \b0$, the functions
$D : \mathbb{S}^{d-1} \to (0,\infty)$ and $G : \mathbb{S}^{d-1} \to \mathbb{S}^{d-1}$, given by
\begin{equation}
D(\bz) = \| g(\bz) \|, \qquad
G(\bz) = \frac{g(\bz)}{\| g(\bz) \|} \;,
\label{eq:DG}
\end{equation}
are well-defined and continuous.
Then
\begin{equation}
g(\bx) = D \!\left( \frac{\bx}{\| \bx \|} \right) G \!\left( \frac{\bx}{\| \bx \|} \right) \| \bx \|, \quad
{\rm for~all~} \bx \ne 0 \;.
\label{eq:g2}
\end{equation}
In this form $g$ is separated into a ``dilation'' towards or away from $\b0$ (as given by $D$)
and a rotation about $\b0$ (as given by $G$).

To investigate the stability of $\b0$ for a continuous, linearly homogeneous map $g$,
it suffices to consider the forward orbits of points $\bz \in \mathbb{S}^{d-1}$.
By \eqref{eq:g2},
\begin{equation}
g^n(\bz) = G^n(\bz) \prod_{i=0}^{n-1} D \!\left( G^i(\bz) \right),
\quad {\rm for~all~} \bz \in \mathbb{S}^{d-1} {\rm ~and~all~} n \ge 1 \;.
\label{eq:gn}
\end{equation}
Thus the attractors of $G$ can help us determine  the behaviour of $g^n(\bz)$ for large values of $n$.
By \eqref{eq:gn},
\begin{equation}
\left\| g^n(\bz) \right\|
= {\rm exp} \!\left[ \sum_{i=0}^{n-1}
\ln \!\left( D \!\left( G^i(\bz) \right) \right) \right],
\label{eq:gn2}
\end{equation}
which motivates the definition
\begin{equation}
\lambda(\mu) = \int_{\mathbb{S}^{d-1}} \ln(D) \,d\mu \;,
\label{eq:lambda}
\end{equation}
for any invariant probability measure $\mu$ of $G$.

In \cite{DoKi08} the above framework was applied to \eqref{eq:g} with $d = 2$.
The authors considered invariant probability measures $\mu$
whose basins have full measure on $\mathbb{S}^{d-1}$.
If $\lambda(\mu) < 0$ then, by \eqref{eq:gn2} and \eqref{eq:lambda},
$g^n(\bz) \to \b0$ for almost all $\bz \in \mathbb{S}^{d-1}$.
Although this is useful from an applied perspective in that
a numerical simulation of a forward orbit from a random initial point will produce $g^n(\bz) \to \b0$ with probability $1$,
it does not imply that $\b0$ is (Lyapunov) stable.

The remainder of this paper is organised as follows.
Notions of the stability of a fixed point are given in \S\ref{sec:stability}.
These are mostly standard and provided for convenience but 
a new notion of ``measure-$\rho$ stability''
(meaning roughly that the fraction of points near the fixed point
whose forward orbit stays near, and converges to, the fixed point is $\rho$) is also introduced.
Section \ref{sec:clh} concerns the class of continuous, linearly homogeneous maps \eqref{eq:linearlyHomogeneous}.
It is shown that if the forward orbits of all points in a neighbourhood of $\b0$ converge to $\b0$,
then the convergence is uniform and $\b0$ is globally asymptotically stable.
It is also shown that $\b0$ remains asymptotically stable under $\co(\bx)$ perturbations to the map.
This helps justify the use of \eqref{eq:g} for studying the stability of $\b0$ for the map \eqref{eq:f}.

In \S\ref{sec:dilation} the stability of $\b0$
for a continuous, linearly homogeneous map $g$ is related to invariant probability measures of $G$.
In \S\ref{sec:fps}, for the piecewise-linear map \eqref{eq:g},
positive eigenvalues of $A_L$ and $A_R$ are related to fixed points of $G$.

The results are then applied to the two-dimensional border-collision normal form \cite{NuYo92},
focusing on the non-invertible case.
Section \ref{sec:theory} provides a comprehensive description of the dynamics of $G$ (in this case a circle map)
and two geometric conditions necessary and sufficient for the asymptotic stability of $\b0$ are given.
In \S\ref{sec:numerics} these conditions are used to numerically identify
an open region of parameter space for which $\b0$ is asymptotically stable
despite each piece of \eqref{eq:g} being area-expanding.
For some parameter values near this stability region
$\b0$ is unstable but $g^n(\bx) \to \b0$ for almost all $\bx \in \mathbb{R}^2$.

Finally \S\ref{sec:conc} provides concluding remarks.
Some proofs are deferred to Appendix \ref{app:proofs}.

Regarding notation, throughout this paper $B_r$ and $\overline{B}_r$ denote
the open and closed balls of radius $r > 0$ centred at $\b0$:
\begin{equation}
B_r = \left\{ \bx \in \mathbb{R}^d ~\middle|~ \| \bx \| < r \right\}, \qquad
\overline{B}_r = \left\{ \bx \in \mathbb{R}^d ~\middle|~ \| \bx \| \le r \right\}.
\label{eq:Br}
\end{equation}
Also ${\rm meas}(X)$ denotes the Lebesgue measure of a set $X \subset \mathbb{R}^d$.

\section{Notions of stability for fixed points}
\label{sec:stability}

Let $F : \mathbb{R}^d \to \mathbb{R}^d$ be continuous and suppose $F(\b0) = \b0$.

\begin{definition}
The fixed point $\b0$ is said to be
\begin{itemize}
\item 
{\em Lyapunov stable} if for all $\ee > 0$ there exists $\delta > 0$ such that
$F^n(\bx) \in B_\ee$ for all $\bx \in B_\delta$ and all $n \ge 0$;
\item 
{\em asymptotically stable} if it is Lyapunov stable and
there exists $\delta > 0$ such that $F^n(\bx) \to \b0$ as $n \to \infty$ for all $\bx \in B_\delta$;
\item
{\em exponentially stable} if it is Lyapunov stable and
there exists $\delta > 0$, $a > 0$ and $0 < b < 1$ such that
$\| F^n(\bx) \| \le a b^n \| \bx \|$ for all $\bx \in B_\delta$ and all $n \ge 0$;
\item
{\em globally asymptotically stable} if it is Lyapunov stable and 
$F^n(\bx) \to \b0$ as $n \to \infty$ for all $\bx \in \mathbb{R}^d$.
\end{itemize}
\label{df:stability}
\end{definition}

The above definitions are given in, for instance, \cite{HaCh08,LaTr02}.
Refer to \cite{LiAn09} for a review of the stability of discrete-time switched systems
where the switching occurs according to a control law
more sophisticated than the sign of $B^{\sf T} \bx$, as in \eqref{eq:f}.

Since our interest is with maps that are non-differentiable,
it is helpful to also consider a weaker measure-theoretic notion of stability.
An invariant set is said to be a Milnor attractor if its basin of attraction
has positive measure \cite{Mi85}.
A Milnor attractor need not be Lyapunov stable,
say if its basin of attraction is {\em riddled}
(that is, every neighbourhood has a positive measure subset that does not belong to the basin) \cite{AlYo92}.
For two-dimensional piecewise-linear continuous maps with two switching manifolds,
Milnor attractors were studied in \cite{Gl01,KaMa99}.

\begin{definition}
For any $\ee,\delta > 0$,
let $A(\ee,\delta)$ be the set of all $\bx \in B_\delta$
for which $F^n(\bx) \to \b0$ as $n \to \infty$ with $F^n(\bx) \in B_\ee$ for all $n \ge 0$.
If the limit
\begin{equation}
\rho = \lim_{\ee \to 0} \lim_{\delta \to 0} \frac{{\rm meas}(A(\ee,\delta))}{{\rm meas}(B_\delta)}
\label{eq:rho}
\end{equation}
exists, then we say that $\b0$ is a {\em measure-$\rho$ stable} fixed point of $F$.
\end{definition}

Essentially $\rho$ is the fraction of points near $\b0$
whose forward orbits stay near $\b0$ and converge to $\b0$.
If $\b0$ is asymptotically stable, then it is measure-$1$ stable.
If $F$ is differentiable and $\b0$ is not Lyapunov stable, then $\b0$ is measure-$0$ stable.
Measure-$1$ stability is analogous to ``almost sure'' stability for Markov processes \cite{BoCo04}.

\section{Continuous linearly homogeneous maps}
\label{sec:clh}

Here we consider a continuous, linearly homogeneous map $g : \mathbb{R}^d \to \mathbb{R}^d$, where $d \ge 1$.
By substituting $\alpha = 0$ into \eqref{eq:linearlyHomogeneous} we see that $\b0$ must be a fixed point of $g$.
We first show that for $\b0$ to be globally asymptotically stable,
we only need that $g^n(\bx) \to \b0$ as $n \to \infty$ for all $\bx$ in a neighbourhood of $\b0$.
We also show that this convergence is uniform\removableFootnote{
I choose not to prove exponential stability here
because it is an immediate consequence of Theorem \ref{th:hots} below and
I would have to repeat it for $f$ in the proof of Theorem \ref{th:hots} anyway.

I choose to prove uniform convergence here
because I need it in the set-up of the proof of Theorem \ref{th:hots}.
Note that we actually prove that $g^n(\overline{B}_r) \to \b0$ uniformly,
and so $B_r$, or in fact any subset of $\overline{B}_r$, converges to $\b0$ uniformly.
}.

\begin{lemma}
Let $g : \mathbb{R}^d \to \mathbb{R}^d$ be continuous and linearly homogeneous.
Suppose there exists $r > 0$ such that $g^n(\bx) \to \b0$ as $n \to \infty$ for all $\bx \in B_r$.
Then
\begin{enumerate}
\item
$\b0$ is a globally asymptotically stable fixed point of $g$, and
\item
$g^n(B_r) \to \b0$ uniformly.
\end{enumerate}
\label{le:clh}
\end{lemma}

A proof in given in Appendix \ref{app:proofs}.
Part (i) is proved by supposing that $\b0$ is not Lyapunov stable
and obtaining a contradiction by constructing a point $\by \in \mathbb{S}^{d-1}$
for which $g^n(\by) \not\to \b0$ as $n \to \infty$.
Part (ii) is proved by showing that on the compact set $\overline{B}_r$
the sequence of functions $\{ g^n \}$ is uniformly bounded and equicontinuous,
from which the result follows by the Arzel\`{a}-Ascoli theorem\removableFootnote{
We use the continuity of $f$ to establish
$\left\| g^n(\bx) - g^n(\by) \right\| < \ee$ for finitely many values of $n$,
and the asymptotic stability of $\b0$ to establish this inequality for arbitrarily large values of $n$.
}.

Next we provide a technical result used below to prove Theorem \ref{th:hots}.

\begin{lemma}
Let $g : \mathbb{R}^d \to \mathbb{R}^d$ be continuous and linearly homogeneous.
Let $f : \mathbb{R}^d \to \mathbb{R}^d$ be continuous with $f(\bx) - g(\bx) = \co(\bx)$.
Then
\begin{equation}
f^n(\bx) - g^n(\bx) = \co(\bx),
\label{eq:figiBound}
\end{equation}
for all $n \ge 1$.
\label{le:figiBound}
\end{lemma}

Lemma \ref{le:figiBound} is proved in Appendix \ref{app:proofs}.
The result requires the linear homogeneity of $g$.
For instance $g(x) = \sqrt{x}$ is not linearly homogeneous
and with $f(x) = \sqrt{x} + x^{\frac{3}{2}}$ we have
$\left\| f(f(x)) - g(g(x)) \right\| = x^{\frac{3}{4}} + \cO \big( x^{\frac{5}{4}} \big)$.
Thus $f(x) - g(x) = \co(x)$ but \eqref{eq:figiBound} does not hold with $n=2$.

\begin{theorem}
Let $g : \mathbb{R}^d \to \mathbb{R}^d$ be continuous and linearly homogeneous.
Let $f : \mathbb{R}^d \to \mathbb{R}^d$ be continuous with $f(\bx) - g(\bx) = \co(\bx)$.
Then there exists $r > 0$ such that $g^n(\bx) \to \b0$ as $n \to \infty$ for all $\bx \in B_r$
if and only if $\b0$ is an exponentially stable fixed point of $f$.
\label{th:hots}
\end{theorem}

\begin{proof} 
First suppose there exists $r > 0$ such that $g^n(\bx) \to \b0$ as $n \to \infty$ for all $\bx \in B_r$.
Choose any $\ee > 0$.
By part (i) of Lemma \ref{le:clh},
$\b0$ is an asymptotically stable fixed point of $g$,
hence there exists $\delta_1 > 0$ such that
\begin{equation}
g^n(\bx) \in B_\ee \;, \quad {\rm for~all~} \bx \in B_{\delta_1} {\rm ~and~all~} n \ge 0 \;,
\label{eq:gLyapunovStable}
\end{equation}
and for all $\bx \in B_{\delta_1}$ we have $g^n(\bx) \to \b0$ as $n \to \infty$.
By part (ii) of Lemma \ref{le:clh},
$g^n \!\left( B_{\delta_1} \right) \to \b0$ uniformly,
hence there exists $N \in \mathbb{Z}$ such that
\begin{equation}
g^n(\bx) \in B_{\frac{\delta_1}{4}} \;, \quad
{\rm for~all~} x \in B_{\delta_1} {\rm ~and~all~} n \ge N \;.
\label{eq:gUniformConvergence}
\end{equation}
Let $R = \frac{\ee}{\delta_1} > 1$.
Since $g$ is linearly homogeneous, we can rewrite \eqref{eq:gLyapunovStable} as
\begin{equation}
\left\| g^n(\bx) \right\| \le R \| \bx \| \;, \quad {\rm for~all~} \bx \in \mathbb{R}^d {\rm ~and~all~} n \ge 0 \;,
\label{eq:gLyapunovStable2}
\end{equation}
and rewrite \eqref{eq:gUniformConvergence} as
\begin{equation}
\left\| g^n(\bx) \right\| \le \frac{1}{4} \| \bx \| \;, \quad
{\rm for~all~} \bx \in \mathbb{R}^d {\rm ~and~all~} n \ge N \;.
\label{eq:gUniformConvergence2}
\end{equation}
By Lemma \ref{le:figiBound}, there exists $\delta > 0$ such that
\begin{equation}
\left\| f^n(\bx) - g^n(\bx) \right\| \le \frac{1}{4} \| \bx \| \;, \quad
{\rm for~all~} \bx \in B_\delta {\rm ~and~all~} n = 1,2,\ldots,N \;,
\label{eq:figiBound2}
\end{equation}
and we assume $\delta < \frac{\delta_1}{2}$.
By \eqref{eq:gLyapunovStable2} and \eqref{eq:figiBound2}, we have
\begin{align}
\left\| f^n(\bx) \right\| \le
\left\| f^n(\bx) - g^n(\bx) \right\| + \left\| g^n(\bx) \right\| \le
\frac{1}{4} \| \bx \| + R \| \bx \| \le 2 R \| \bx \|, & \nonumber \\
{\rm for~all~} \bx \in B_\delta {\rm ~and~all~} n = 0,1,\ldots,N \;. &
\label{eq:fiBound}
\end{align}
Also by \eqref{eq:gUniformConvergence2} and \eqref{eq:figiBound2} with $n = N$, we have
\begin{align}
\left\| f^N(\bx) \right\| \le
\left\| f^N(\bx) - g^N(\bx) \right\| + \left\| g^N(\bx) \right\| \le
\frac{1}{4} \| \bx \| + \frac{1}{4} \| \bx \| = \frac{1}{2} \| \bx \|, \quad
{\rm for~all~} \bx \in B_\delta \;.
\label{eq:fNBound}
\end{align}
By applying \eqref{eq:fiBound} and \eqref{eq:fNBound} recursively,
we first see that for all $\bx \in B_\delta$ and all $n \ge 0$, we have
$\left\| f^n(\bx) \right\| \le 2 R \delta < \ee$,
which verifies Lyapunov stability.
Second, for all $\bx \in B_\delta$, $j \ge 0$ and $n \ge j N$ we have
$\left\| f^n(\bx) \right\| \le \left( \frac{1}{2} \right)^j 2 R \| \bx \|$.
Thus,
\begin{equation}
\left\| f^n(\bx) \right\| \le 2 R \left( 2^{\frac{-1}{N}} \right)^n \| \bx \|, \quad
{\rm for~all~} \bx \in B_\delta {\rm ~and~all~} n \ge 0 \;,
\nonumber
\end{equation}
which verifies exponential stability.

Conversely suppose $\b0$ is an exponentially stable fixed point of $f$.
Then there exists $\delta > 0$, $a > 0$ and $0 < b < 1$ such that
$\left\| f^n(\bx) \right\| \le a b^n \| \bx \|$ for all $\bx \in B_\delta$ and all $n \ge 0$.
Let $N \in \mathbb{Z}$ be such that $a b^N < \frac{1}{4}$.
Then
\begin{equation}
\left\| f^N(\bx) \right\| \le \frac{1}{4} \| \bx \|, \quad
{\rm for~all~} \bx \in B_\delta \;.
\label{eq:expStableConverseProof1}
\end{equation}
By Lemma \ref{le:figiBound} there exists $r > 0$ such that
\begin{equation}
\left\| f^N(\bx) - g^N(\bx) \right\| \le \frac{1}{4} \| \bx \|, \quad
{\rm for~all~} \bx \in B_r \;,
\label{eq:expStableConverseProof2}
\end{equation}
and assume $r \le \delta$.
Then, by \eqref{eq:expStableConverseProof1} and \eqref{eq:expStableConverseProof2},
\begin{align*}
\left\| g^N(\bx) \right\| \le
\left\| f^N(\bx) - g^N(\bx) \right\| + \left\| f^N(\bx) \right\| \le
\frac{1}{4} \| \bx \| + \frac{1}{4} \| \bx \| = \frac{1}{2} \| \bx \|, \quad
{\rm for~all~} \bx \in B_r \;.
\end{align*}
Thus $g^n(\bx) \to \b0$ as $n \to \infty$ for all $\bx \in B_r$ as required.
\end{proof}

A simple consequence of Theorem \ref{th:hots} is that if $\b0$ is asymptotically stable for $g$
then $\b0$ is also asymptotically stable for $f$.
The converse of this statement is not true.
For example, $0$ is an asymptotically stable fixed point of $f(x) = x - x^3$,
but $0$ is not an asymptotically stable fixed point of $g(x) = x$.
The converse may be true if asymptotic stability is replaced by Lyapunov stability:

\begin{conjecture}
Let $g : \mathbb{R}^d \to \mathbb{R}^d$ be continuous and linearly homogeneous.
Let $f : \mathbb{R}^d \to \mathbb{R}^d$ be continuous with $f(\bx) - g(\bx) = \co(\bx)$.
Suppose $\b0$ is an unstable (i.e.~not Lyapunov stable) fixed point of $g$.
Then $\b0$ is an unstable fixed point of $f$.
\label{cj:hots2}
\end{conjecture}

Conjecture \ref{cj:hots2} is expected to be true because $g$ is essentially linear in radial directions.
For this reason an orbit of $g$ that heads away from $\b0$ is expected to
do so for $f$ at the same asymptotic rate.
This is effectively claiming that an unstable manifold of $g$ persists for $f$.
It remains to be seen if the stable manifold theorem can be generalised to
the present situation involving non-differentiable maps.

\section{Average dilation on invariant probability measures}
\label{sec:dilation}

Here we consider a continuous, linearly homogeneous map $g : \mathbb{R}^d \to \mathbb{R}^d$, where $d \ge 2$.
We assume $g(\bx) = \b0$ only for $\bx = \b0$ so that $D$ and $G$, as defined by \eqref{eq:DG},
are well-defined and continuous.

To describe the size of subsets of $\mathbb{R}^d$ and $\mathbb{S}^{d-1}$ we use the most natural measures available for these spaces.
Specifically, on $\mathbb{R}^d$ we use the Lebesgue measure and on $\mathbb{S}^{d-1}$ we use the {\em spherical measure}.
For any $\Omega \subset \mathbb{S}^{d-1}$, the spherical measure of $\Omega$ is
\begin{equation}
{\rm sph~meas}(\Omega) =
\frac{{\rm meas} \!\left( \{ \alpha \bz \,|\, \bz \in \Omega, 0 \le \alpha \le 1 \} \right)}{{\rm meas}(B_1)} \;.
\label{eq:sphericalMeasure}
\end{equation}

We begin by considering physical measures of $G$.
A physical measure is an invariant probability measure whose basin has positive measure.
The basin of an invariant probability measure $\mu$ of $G$ is the set of all $\bz \in \mathbb{S}^{d-1}$ for which
\begin{equation}
\lim_{n \to \infty} \frac{1}{n} \sum_{i=0}^{n-1} \varphi \!\left( G^i(\bz) \right) =
\int_{\mathbb{S}^{d-1}} \varphi \,d\mu \;,
\label{eq:basin}
\end{equation}
for all continuous $\varphi : \mathbb{S}^{d-1} \to \mathbb{R}$.

\begin{theorem}
Let $g : \mathbb{R}^d \to \mathbb{R}^d$ be continuous and linearly homogeneous with
$g(\bx) = \b0$ only for $\bx = \b0$.
Let $\mu_1,\ldots,\mu_m$ be a collection of physical measures of $G$, given by \eqref{eq:DG},
and let $\eta_1,\ldots,\eta_m$ denote the (spherical) measures of their basins.
Suppose $\sum_{j=1}^m \eta_j = 1$ (i.e.~the union of the basins has full measure on $\mathbb{S}^{d-1}$)
and $\lambda(\mu_j) \ne 0$ for each $j \in \{ 1,\ldots,m \}$.
Then $\b0$ is a measure-$\rho$ stable fixed point of $g$, where
\begin{equation}
\rho = \sum_{\lambda(\mu_j) < 0} \eta_j \;.
\label{eq:rhoPhysical}
\end{equation}
\label{th:physical}
\end{theorem}

\begin{proof}
For each $j \in \{ 1,\ldots,m \}$, let $\Psi_j \subset \mathbb{S}^{d-1}$ denote the basin of $\mu_j$.
For any $\bx \in \Psi_j$, by \eqref{eq:gn2} putting $\varphi = \ln(D)$ into \eqref{eq:basin} gives
$\lim_{n \to \infty} \ln \!\left( \left\| g^n(\bz) \right\|^{\frac{1}{n}} \right) =
\int_{\mathbb{S}^{d-1}} \ln(D) \,d\mu_j$.
That is,
\begin{equation}
\lim_{n \to \infty} \left\| g^n(\bz) \right\|^{\frac{1}{n}} = {\rm e}^{\lambda(\mu_j)} \;.
\label{eq:physicalProof1}
\end{equation}
Let $Q_{\rm neg} \subset \{ 1,\ldots,m \}$ denote the set of all $j \in \{ 1,\ldots,m \}$
for which $\lambda(\mu_j) < 0$, and let
$Q_{\rm pos} = \{ 1,\ldots,m \} \hspace{-.5mm}\setminus\! Q_{\rm neg}$.
Let
\begin{align*}
X_{\rm neg} &= \bigcup_{j \in Q_{\rm neg}}
\left\{ \alpha \bz ~\middle|~ \bz \in \Psi_j, 0 \le \alpha \le 1 \right\}, \\
X_{\rm pos} &= \bigcup_{j \in Q_{\rm pos}}
\left\{ \alpha \bz ~\middle|~ \bz \in \Psi_j, 0 \le \alpha \le 1 \right\}.
\end{align*}
Choose any $\bx \in X_{\rm neg}$.
Then $\bx = \alpha \bz$ for some $\bz \in \Psi_j$ and $0 \le \alpha \le 1$
where $j \in Q_{\rm neg}$.
By \eqref{eq:physicalProof1} $g^n(\bz) \to \b0$ as $n \to \infty$ because $\lambda(\mu_j) < 0$.
Then $g^n(\bx) \to \b0$ as $n \to \infty$ by the linear homogeneity of $g$.
Similarly $\left\| g^n(\bx) \right\| \to \infty$ as $n \to \infty$ for any $\bx \in X_{\rm pos}$.

By the linear homogeneity of $g$, \eqref{eq:rho} can be rewritten as
\begin{equation}
\rho = \lim_{N \to \infty} \frac{{\rm meas}(A(N,1))}{{\rm meas}(B_1)} \;,
\label{eq:rho2}
\end{equation}
where $A(N,1)$ is the set of all $\bx \in B_1$ for which $g^n(\bx) \to \b0$ as $n \to \infty$ with
$g^n(\bx) \in B_N$ for all $n \ge 0$.
The set $\lim_{N \to \infty} A(N,1)$ contains all points in $X_{\rm neg}$ and no points in $X_{\rm pos}$.
Thus by \eqref{eq:sphericalMeasure} we have $\rho = \sum_{j \in X_{\rm neg}} \eta_j$ as required.
\end{proof}

Equation \eqref{eq:rhoPhysical} gives the value of $\rho$ as the sum of all $\mu_j$ for which $\lambda(\mu_j) < 0$.
For any $\bz \in \mathbb{S}^{d-1}$ in the basin of $\mu_j$,
if $\lambda(\mu_j) < 0$ then $g^n(\bz) \to \b0$ as $n \to \infty$,
while if $\lambda(\mu_j) > 0$ then $\left\| g^n(\bz) \right\| \to \infty$ as $n \to \infty$.
For this reason we can obtain further insight by considering the space $\mathbb{R}^d$
together with the point at infinity.
If $\rho > 0$ then $\b0$ is a Milnor attractor,
while if $\rho < 1$ then the point at infinity is a Milnor attractor.
With the assumptions of Theorem \ref{th:physical}
these are the only two Milnor attractors possible for $g$.

Theorem \ref{th:physical} is practical as there are rarely a large number of physical measures
(indeed often there is just one).
The value of $\rho$ can be calculated
by evaluating $\lambda(\mu_j)$ for each physical measure $\mu_j$.
However, for asymptotic stability the following result refers us to ergodic measures
of which there may be uncountably many.

\begin{theorem}
Let $g : \mathbb{R}^d \to \mathbb{R}^d$ be continuous and linearly homogeneous with
$g(\bx) = \b0$ only for $\bx = \b0$,
and let $G$ be given by \eqref{eq:DG}.
\begin{enumerate}
\item
If $\lambda(\mu) < 0$ for every ergodic invariant probability measure $\mu$ of $G$,
then $\b0$ is an asymptotically stable fixed point of $g$.
\item
If $\b0$ is a Lyapunov stable fixed point of $g$,
then $\lambda(\mu) \le 0$ for every ergodic invariant probability measure $\mu$ of $G$.
\end{enumerate}
\label{th:dilation}
\end{theorem}

Part (i) of Theorem \ref{th:dilation} is more difficult to prove than Theorem \ref{th:physical}
because we must consider all $\bz \in \mathbb{S}^{d-1}$, not just almost all $\bz \in \mathbb{S}^{d-1}$.
Below we first consider the sequence of time-average probability measures
$\mu_n = \frac{1}{n} \sum_{i=0}^{n-1} \delta_{G^i(\bz)}$,
where we use $\delta_\by$ to denote the Dirac probability measure supported at a point $\by \in \mathbb{S}^{d-1}$.
We then identify a suitable limit point $\mu$ of $\{ \mu_n \}$ and
use the ergodic decomposition theorem to express $\mu$ as an average of ergodic measures.

\begin{proof}
Let $\mathcal{M}_G$ denote the set of all invariant probability measures of $G$
and let $\mathcal{E}_G \subset \mathcal{M}_G$ denote the set of all such measures that are ergodic.
We prove the two parts of the theorem in order.
\begin{enumerate}
\item
Choose any $\bz \in \mathbb{S}^{d-1}$.
We will show that $g^n(\bz) \to \b0$ as $n \to \infty$,
and so by linear homogeneity $g^n(\bx) \to \b0$ as $n \to \infty$ for all $\bx \in \mathbb{R}^d$.
Then by Lemma \ref{le:clh}(i), $\b0$ must be an asymptotically stable fixed point of $g$.

For all $n \ge 1$, $\mu_n = \frac{1}{n} \sum_{i=0}^{n-1} \delta_{G^i(\bz)}$
is a probability measure on $\mathbb{S}^{d-1}$.
By \eqref{eq:gn2},
\begin{align}
\left\| g^n(\bz) \right\|^{\frac{1}{n}} &=
{\rm exp} \!\left[ \frac{1}{n} \sum_{i=0}^{n-1}
\ln \!\left( D \!\left( G^i(\bz) \right) \right) \right] \nonumber \\
&={\rm exp} \!\left[ \frac{1}{n} \sum_{i=0}^{n-1}
\int_{\mathbb{S}^{d-1}} \ln(D) \,d\delta_{G^i(z)} \right] \nonumber \\
&={\rm exp} \!\left[ \int_{\mathbb{S}^{d-1}} \ln(D) \,d\mu_n \right]. \nonumber
\end{align}
Since the sequence $\{ \mu_n \}$ may be non-convergent,
we consider the set of all limit points of $\{ \mu_n \}$, call it $\Xi$.
The set $\Xi$ is non-empty and weak-compact \cite{DeGr76}\removableFootnote{
See Proposition 3.8 of \cite{DeGr76}.
In my notation, they prove that $\Xi$ is closed in the set of all invariant measures of $G$
(in the proof they say ``Obviously $V_T(\mu)$ is closed'',
presumably because $\Xi$ is defined as a set of limit points and so must be closed).
Since the set of all invariant probability measures of $G$ is bounded, $\Xi$ is bounded and hence weak*-compact
(although I don't really understand why weak*-compact instead of just compact).

Having thought more, this should be weak-compact rather than weak*-compact
(because I am doing measure theory and not functional analysis/topology).
By weak-compact, we mean that $\Xi$ is bounded and closed (i.e.~contains its limit points) in the weak topology.
That is, for every sequence $\xi_n$ in $\Xi$
that converges weakly, i.e.~$\int f \,d\xi_n \to \int f \,d\xi$ for every bounded continuous $f$,
we have $\xi \in \Xi$.
},
thus there exists $\mu \in \Xi$ such that
\begin{equation}
\limsup_{n \to \infty} \left\| g^n(\bz) \right\|^{\frac{1}{n}} =
{\rm exp} \!\left[ \int_{\mathbb{S}^{d-1}} \ln(D) \,d\mu \right].
\nonumber
\end{equation}
Since $\Xi \subset \mathcal{M}_G$ \cite{Wa82}\removableFootnote{
Specifically Theorem 6.9.
},
$\mu$ is an invariant probability measure of $G$ and so
\begin{equation}
\limsup_{n \to \infty} \left\| g^n(\bz) \right\|^{\frac{1}{n}} = {\rm e}^{\lambda(\mu)} \;.
\label{eq:dilationProof3}
\end{equation}

By the ergodic decomposition theorem \cite{KaMc10,Wa82}
there exists a unique measure $\tau$ on $\mathcal{M}_G$
with $\tau \!\left( \mathcal{E}_G \right) = 1$ such that
$\int_{\mathbb{S}^{d-1}} \varphi \,d\mu =
\int_{\mathcal{E}_G} \int_{\mathbb{S}^{d-1}} \varphi \,d\sigma \,d\tau(\sigma)$
for all continuous $\varphi : \mathbb{S}^{d-1} \to \mathbb{R}$.
Putting $\varphi = \ln(D)$ gives
\begin{equation}
\lambda(\mu) = \int_{\mathcal{E}_G} \lambda(\sigma) \,d\tau(\sigma).
\label{eq:dilationProof4}
\end{equation}
By assumption $\lambda(\sigma) < 0$ for every $\sigma \in \mathcal{E}_G$.
Since $\mathcal{E}_G$ is compact
and $\lambda(\sigma)$ is a continuous function on $\mathcal{E}_G$\removableFootnote{
Which is because $\sigma_n \to \sigma$ strongly implies $\sigma_n \to \sigma$ weakly.
},
there exists $a > 0$ such that $\lambda(\sigma) < -a$ for all $\sigma \in \mathcal{E}_G$.
By \eqref{eq:dilationProof4}, $\lambda(\mu) \le -a < 0$,
and so by \eqref{eq:dilationProof3} we have $g^n(\bz) \to \b0$ as $n \to \infty$.
\item
Suppose (for a proof by contrapositive) that $\lambda(\mu) > 0$ for some $\mu \in \mathcal{E}_G$.
Let $\bz \in \mathbb{S}^{n-1}$ be a point in the basin of $\mu$
(this basin is non-empty because $\mu$ is ergodic \cite{Wa82})\removableFootnote{
Such a $\bz$ exists by Lemma 6.13 of \cite{Wa82}.
By ``the'' ergodic theorem, one could in fact take almost any (with respect to $\mu$) $\bz$,
see Theorem 6.14 of \cite{Wa82}.
}.
As in the proof of Theorem \ref{th:physical}, putting $\varphi = \ln(D)$ into \eqref{eq:basin} leads to 
$\lim_{n \to \infty} \left\| g^n(\bz) \right\|^{\frac{1}{n}} = {\rm e}^{\lambda(\mu)}$,
and hence $\left\| g^n(\bz) \right\| \to \infty$ as $n \to \infty$.
Since $g$ is linearly homogeneous,
$\left\| g^n(\alpha \bz) \right\| \to \infty$ as $n \to \infty$ for all $\alpha > 0$.
Thus $\b0$ is not a Lyapunov stable fixed point of $g$.
\end{enumerate}
\end{proof}

\section{The relationship between fixed points of $G$ and positive eigenvalues of $A_L$ and $A_R$}
\label{sec:fps}

Here we consider the piecewise-linear approximation \eqref{eq:g} with $d \ge 2$.
As long as $0$ is not an eigenvalue of $A_L$ or $A_R$,
we have $g(\bx) = \b0$ only for $\bx = \b0$ and so $D$ and $G$ are well-defined and continuous and given by
\begin{equation}
D(\bz) = \begin{cases}
\| A_L \bz \|, & B^{\sf T} \bz \le 0 \;, \\
\| A_R \bz \|, & B^{\sf T} \bz \ge 0 \;,
\end{cases} \qquad
G(\bz) = \begin{cases}
\frac{A_L \bz}{\| A_L \bz \|} \;, & B^{\sf T} \bz \le 0 \;, \\
\frac{A_R \bz}{\| A_R \bz \|} \;, & B^{\sf T} \bz \ge 0 \;.
\end{cases}
\label{eq:DG2}
\end{equation}

Here we show how positive eigenvalues of $A_L$ and $A_R$
can be matched one-to-one with fixed points of $G$\removableFootnote{
I feel that Lemma \ref{le:fps} is written remarkably concisely.
I tried writing Lemmas \ref{le:fps} and \ref{le:ray} as a single lemma
but the result was longer than the concatenation of Lemmas \ref{le:fps} and \ref{le:ray}.
}.
More generally a periodic solution of $G$ corresponds to an eigenvector
of a matrix formed by multiplying copies of $A_L$ and $A_R$
in the order determined by the itinerary of the periodic solution relative to the switching manifold \cite{Si16}.

\begin{lemma}
A point $\bz^* \in \mathbb{S}^{d-1}$ is a fixed point of $G$, given by \eqref{eq:DG2},
if and only if it is an eigenvector of $A_L$ or $A_R$ corresponding to an eigenvalue $\lambda > 0$.
Moreover, $\lambda = D(\bz^*)$.
\label{le:fps}
\end{lemma}

\begin{proof}
If $\bz^*$ is a fixed point of $G$,
then $A_J \bz^* = \| A_J \bz^* \| \bz^*$, where $J$ is either $L$ or $R$.
Thus $\bz^*$ is an eigenvector of $A_J$ corresponding to the positive eigenvalue $D(\bz^*)$.

Conversely, suppose $A_J \bz^* = \lambda \bz^*$, where $J$ is either $L$ or $R$ and $\lambda > 0$.
We may assume $B^{\sf T} \bz^* \le 0$ if $J = L$ and $B^{\sf T} \bz^* \ge 0$ if $J = R$
in view of the replacement $\bz^* \mapsto -\bz^*$.
Then $D(\bz^*) = \lambda$ and so $G(\bz^*) = \bz^*$.
\end{proof}

\begin{lemma}
Let $\bz^* \in \mathbb{S}^{d-1}$ be a fixed point of $G$, given by \eqref{eq:DG2}.
Then $g(\bx) = \lambda \bx$ for any $\bx \in \left\{ \alpha \bz^* ~\middle|~ \alpha \ge 0 \right\}$,
where $\lambda = D(\bz^*)$.
\label{le:ray}
\end{lemma}

\begin{proof}
Choose any $\alpha \ge 0$ and let $\bx = \alpha \bz^*$.
By linear homogeneity, $g(\bx) = g(\alpha \bz^*) = \alpha g(\bz^*)$.
Using \eqref{eq:g2}, $g(\bz^*) = D(\bz^*) G(\bz^*) = \lambda \bz^*$.
Thus $g(\bx) = \lambda \alpha \bz^* = \lambda \bx$.
\end{proof}

\section{Theoretical results for two-dimensional non-invertible maps}
\label{sec:theory}

If the piecewise-linear approximation \eqref{eq:g} satisfies certain non-degeneracy conditions,
we can apply a linear coordinate change so that $A_L$ and $A_R$ are converted to companion matrices and
$B$ is converted to the first standard basis vector \cite{Di03,DiBu08,Si16}.
This produces the border-collision normal form evaluated at the border-collision bifurcation.
In two dimensions, $\bx = (x,y)$, this map may be written as
\begin{equation}
g(x,y) =
\begin{cases}
\begin{bmatrix} \tau_L & 1 \\ -\delta_L & 0 \end{bmatrix}
\begin{bmatrix} x \\ y \end{bmatrix}, & x \le 0 \;, \\
\begin{bmatrix} \tau_R & 1 \\ -\delta_R & 0 \end{bmatrix}
\begin{bmatrix} x \\ y \end{bmatrix}, & x \ge 0 \;,
\end{cases}
\label{eq:g2d}
\end{equation}
where $\tau_L,\delta_L,\tau_R,\delta_R \in \mathbb{R}$ are constants.

Here we study $g$, given by \eqref{eq:g2d}, in the case that it is non-invertible,
that is $\delta_L \delta_R \le 0$.
Ignoring the special case that either $\delta_L$ or $\delta_R$ is zero,
we may assume without loss of generality that
\begin{equation}
\delta_L > 0 \;, \qquad \delta_R < 0 \;.
\label{eq:deltaLR}
\end{equation}
Other studies of two-dimensional, non-invertible, piecewise-linear maps include \cite{GlWo09,GlWo11,MiGa96}.

With \eqref{eq:deltaLR}, the range of $g$ is the set of all points $(x,y)$ with $y \ge 0$.
The boundary of the range, $y = 0$,
is the image of the switching manifold, $x = 0$.

Since $g$ is linearly homogeneous, it is convenient to work in polar coordinates
\begin{equation}
x = r \cos(\theta), \qquad y = r \sin(\theta).
\label{eq:polar}
\end{equation}
Since the range of $g$ is $y \ge 0$, and no point in the range maps to the negative $x$-axis,
we consider only $\theta \in [0,\pi)$ unless otherwise stated.
This allows us to write $\theta = \tan^{-1} \!\left( \frac{y}{x} \right)$ without ambiguity
(ignoring the fixed point $(x,y) = (0,0)$, and if $x=0$ we mean this to give $\theta = \frac{\pi}{2}$).

Next we write $D$ and $G$ in terms of $\theta$ so that they are explicitly one-dimensional.
We have
\begin{equation}
D(\theta) = \begin{cases}
\sqrt{\left( \tau_R \cos(\theta) + \sin(\theta) \right)^2 + \delta_R^2 \cos^2(\theta)} \;, &
0 \le \theta \le \frac{\pi}{2} \;, \\
\sqrt{\left( \tau_L \cos(\theta) + \sin(\theta) \right)^2 + \delta_L^2 \cos^2(\theta)} \;, &
\frac{\pi}{2} \le \theta < \pi \;,
\end{cases}
\label{eq:D2d}
\end{equation}
and
\begin{equation}
G(\theta) = \begin{cases}
\tan^{-1} \!\left( \frac{-\delta_R}{\tau_R + \tan(\theta)} \right), &
0 \le \theta \le \frac{\pi}{2} \;, \\
\tan^{-1} \!\left( \frac{-\delta_L}{\tau_L + \tan(\theta)} \right), &
\frac{\pi}{2} \le \theta < \pi \;,
\end{cases}
\label{eq:G2d}
\end{equation}
see Fig.~\ref{fig:ppPolar1}-A.

\begin{figure}[b!]
\begin{center}
\setlength{\unitlength}{1cm}
\begin{picture}(15.6,8)
\put(0,0){\includegraphics[width=8cm]{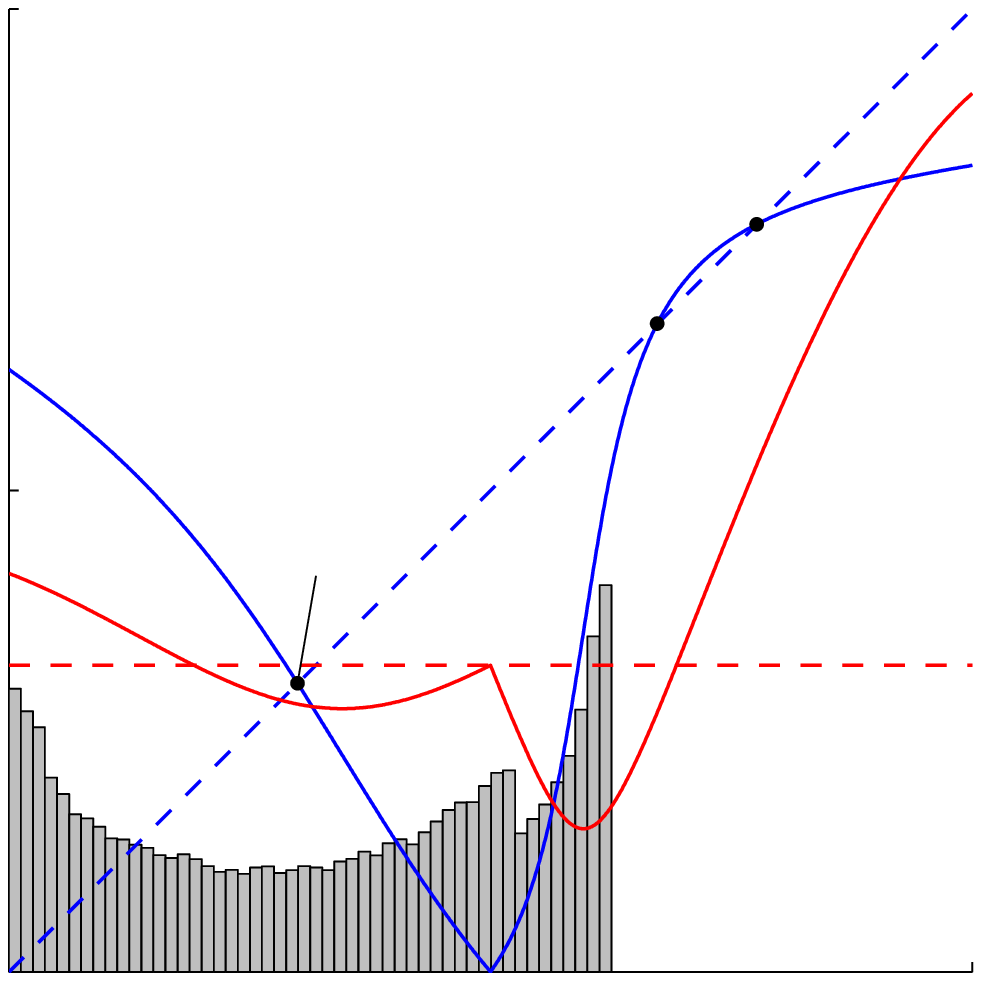}}
\put(8.5,.3){\includegraphics[width=7cm]{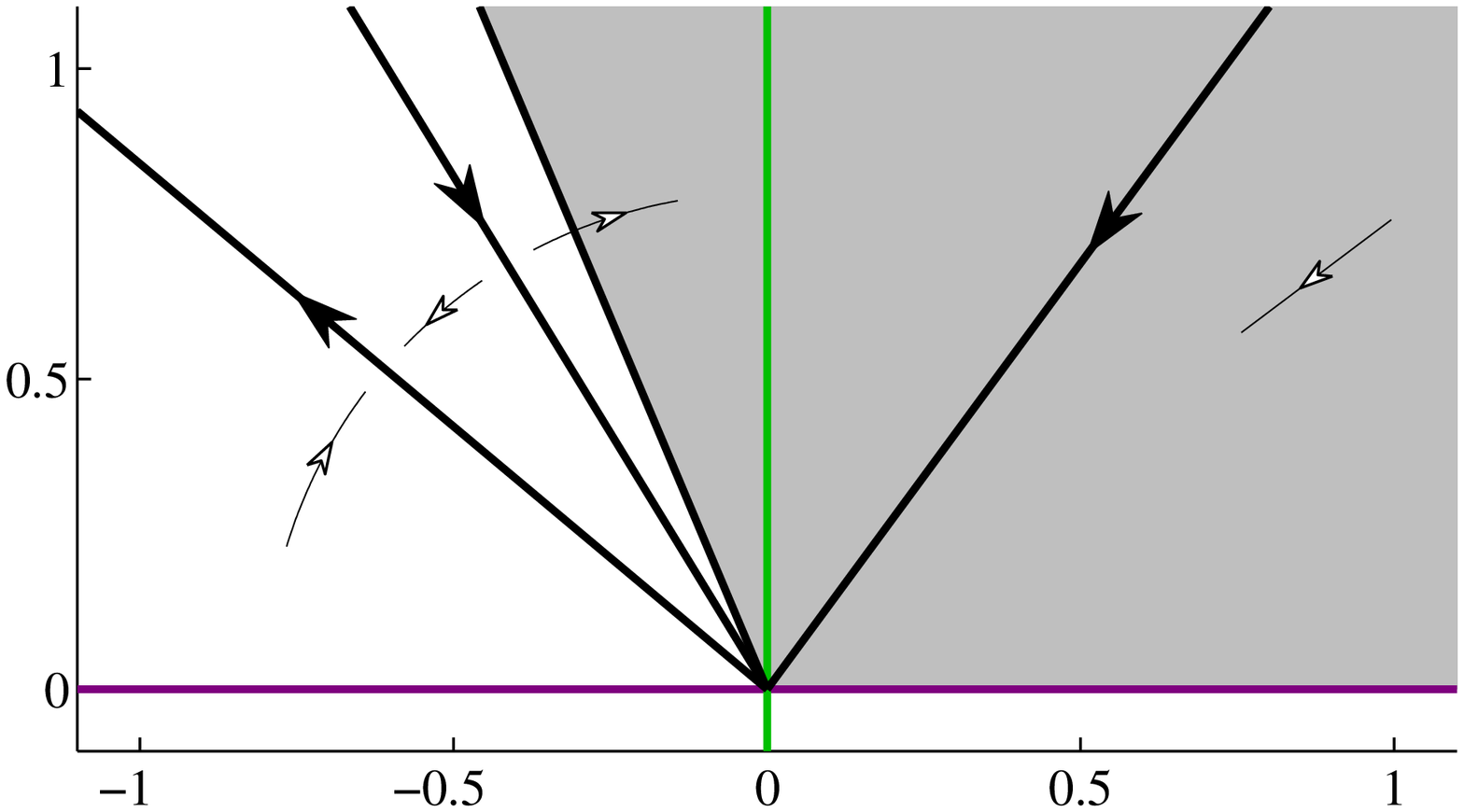}}
\put(1.14,7.85){\sf \bfseries A}
\put(9.41,4.4){\sf \bfseries B}
\put(5.4,.32){\small $\theta$}
\put(.73,.54){\scriptsize $0$}
\put(4.23,.54){\scriptsize $\frac{\pi}{2}$}
\put(7.73,.54){\scriptsize $\pi$}
\put(.58,.74){\scriptsize $0$}
\put(.52,4.26){\scriptsize $\frac{\pi}{2}$}
\put(.58,7.79){\scriptsize $\pi$}
\put(3,3.79){\scriptsize $\theta^R_+$}
\put(5.12,5.58){\scriptsize $\theta^L_-$}
\put(5.95,6.45){\scriptsize $\theta^L_+$}
\put(1.33,3.56){\scriptsize $D(\theta)$}
\put(1.33,4.88){\scriptsize $G(\theta)$}
\put(7,3.15){\tiny $D\!=\!1$}
\put(6.84,7.47){\tiny $G\!=\!\theta$}
\put(12.14,.26){\small $x$}
\put(8.6,3.25){\small $y$}
\put(14.83,1.18){\tiny $\theta\!=\!0$}
\put(12.32,4.06){\tiny $\theta\!=\!\frac{\pi}{2}$}
\put(11.09,4.06){\tiny $\theta\!=\!\theta_\Lambda$}
\put(9.23,1.18){\tiny $\theta\!=\!\pi$}
\put(14.5,3.07){\footnotesize $\Lambda$}
\put(14.46,3.88){\scriptsize $\gamma^R_+$}
\put(10.03,3.88){\scriptsize $\gamma^L_-$}
\put(9.42,2.9){\scriptsize $\gamma^L_+$}
\end{picture}
\caption{
Panel A shows $D$ and $G$, given by \eqref{eq:D2d} and \eqref{eq:G2d},
with $(\tau_L,\delta_L,\tau_R,\delta_R) = (2.5,1.4,-0.5,-1.2)$.
We also show the fixed points of $G$ and the probability density function of a
physical measure of $G$ (more precisely we show a histogram of the first
$10^5$ iterates of $\theta = 0$ under $G$).
Panel B illustrates the dynamics of $g$, given by \eqref{eq:g2d},
in the $(x,y)$-plane for the same combination of parameter values.
The solid arrows indicate the direction of forward iteration of $g$
on the invariant rays $\gamma^R_+$, $\gamma^L_-$, and $\gamma^L_+$.
The hollow arrows indicate the rough direction of other iterates of $g$.
Specifically, iterates are attracted to $\gamma^L_+$,
are repelled from $\gamma^L_-$,
and approach $\b0$ in $\Lambda$ \eqref{eq:Lambda}.
\label{fig:ppPolar1}
} 
\end{center}
\end{figure}

In view of Lemma \ref{le:fps}, in order to identify fixed points of $G$ we look at
the eigenvalues and eigenvectors of $A_L$ and $A_R$.
These are given by
\begin{equation}
\lambda^J_\pm = \frac{\tau_J}{2} \pm \sqrt{\frac{\tau_J^2}{4} - \delta_J} \;, \qquad
v^J_\pm = \begin{bmatrix} 1 \\ \lambda^J_\pm - \tau_J \end{bmatrix},
\label{eq:lambdavJpm}
\end{equation}
where $J = L,R$.
If $\lambda^J_\pm > 0$, then $v^J_\pm$ corresponds to the angle
\begin{equation}
\theta^J_\pm = \tan^{-1} \left( \lambda^J_\pm - \tau_J \right),
\label{eq:thetaJpm}
\end{equation}
and we let
\begin{equation}
\gamma^J_\pm = \left\{ (x,y) ~\middle|~ r \ge 0,\, \theta = \theta^J_\pm \right\},
\label{eq:gammaJpm}
\end{equation}
denote the corresponding invariant ray of $g$ described in Lemma \ref{le:ray}.

Next we provide two lemmas that characterise the dynamics of $G$ in
$\left( \frac{\pi}{2}, \pi \right)$ and $\left[ 0, \frac{\pi}{2} \right]$ respectively.
These results can be used to infer the dynamics of $g$ in the half-planes $x \le 0$ and $x \ge 0$.
The proofs, given in Appendix \ref{app:proofs}, are straight-forward because it suffices to
work with the linear maps $A_L \bx$ and $A_R \bx$.

\begin{lemma}
Suppose $\delta_L > 0$ for the map $G$ given by \eqref{eq:G2d}.
\begin{enumerate}
\item
If $\tau_L < 2 \sqrt{\delta_L}$,
then $G$ has no fixed points in $\left( \frac{\pi}{2}, \pi \right)$
and for all $\theta \in \left( \frac{\pi}{2}, \pi \right)$
there exists $i > 0$ such that $G^i(\theta) \in \left[ 0, \frac{\pi}{2} \right]$.
\item
If $\tau_L > 2 \sqrt{\delta_L}$,
then $G$ has two fixed points in $\left( \frac{\pi}{2}, \pi \right)$, namely $\theta^L_+$ and $\theta^L_-$.
Moreover, for all $\theta \in \left( \frac{\pi}{2}, \theta^L_- \right)$
there exists $i > 0$ such that $G^i(\theta) \in \left[ 0, \frac{\pi}{2} \right]$,
and for all $\theta \in \left( \theta^L_-, \pi \right)$
we have $G^n(\theta) \to \theta^L_+$ as $n \to \infty$.
\end{enumerate}
\label{le:GL}
\end{lemma}

\begin{lemma}
Suppose $\delta_R < 0$ for the map $G$ given by \eqref{eq:G2d}.
On $\left[ 0, \frac{\pi}{2} \right]$, $G$ has the unique fixed point $\theta^R_+$.
\begin{enumerate}
\item
If $\tau_R > 0$, then $\theta^R_+$ is a globally attracting fixed point of $G$
on the forward invariant interval $\left[ 0, \frac{\pi}{2} \right]$.
\item
If $\tau_R < 0$, then $\theta^R_+$ is a repelling fixed point of $G$.
\end{enumerate}
\label{le:GR}
\end{lemma}

Let
\begin{equation}
\theta_\Lambda = G(0) = \tan^{-1} \!\left( \frac{-\delta_R}{\tau_R} \right),
\label{eq:thetaLambda}
\end{equation}
and
\begin{equation}
\Lambda = \left\{ (x,y) ~\middle|~ r \ge 0,\, 0 \le \theta \le \theta_\Lambda \right\}.
\label{eq:Lambda}
\end{equation}
The sector $\Lambda$ is the image of the closure of the first quadrant of the $(x,y)$-plane under $g$.
The following result is proved in Appendix \ref{app:proofs}.

\begin{lemma}
Suppose $\delta_L > 0$ and $\delta_R < 0$ for the map $g$ given by \eqref{eq:g2d}.
Then $\Lambda$ is a forward invariant set of $g$ unless
$\tau_R < 0$, $\tau_L > 2 \sqrt{\delta_L}$, and
$\theta^L_- < \theta_\Lambda < \theta^L_+$.
Moreover, if $\tau < 2 \sqrt{\delta_L}$ then there exists $M \in \mathbb{Z}$ such that
$g^M(x,y) \in \Lambda$ for all $(x,y) \in \mathbb{R}^2$.
\label{le:Lambda}
\end{lemma}

If $\tau_R > 0$ then $\Lambda$ is forward invariant by Lemma \ref{le:Lambda}.
Also $\Lambda$ is contained in the half-plane $x \ge 0$ and so by Lemma \ref{le:GR}
iterates of $g$ in $\Lambda$ simply approach the ray $\gamma^R_+$
and thus tend to $\b0$ if $\lambda^R_+ < 1$, and diverge if $\lambda^R_+ > 1$.
If $\tau_R < 0$ then $\Lambda$ involves points on both sides of the switching manifold
and so the dynamics of $g$ within $\Lambda$ can be complicated (indeed $G$ may be chaotic on $[0,\theta_\Lambda]$).

We complete this section by providing necessary and sufficient conditions for $\b0$ to be an asymptotically stable
fixed point of $g$ in the case $\tau_L < 2 \sqrt{\delta_L}$ (which includes the interesting case $\tau_R < 0$).
This was obtained by taking a result of Gardini \cite{Ga92}, for a piecewise-linear macro-economic market model,
and adapting it to the normal form \eqref{eq:g2d}.

\begin{figure}[b!]
\begin{center}
\setlength{\unitlength}{1cm}
\begin{picture}(7,4)
\put(0,0){\includegraphics[width=7cm]{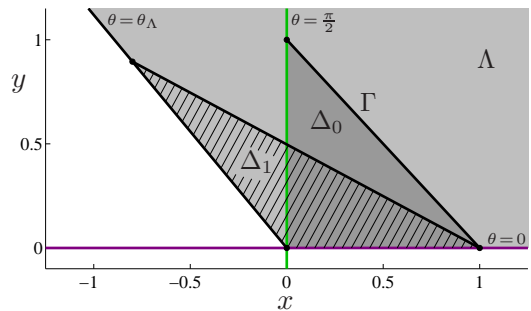}}
\put(3.64,0){\small $x$}
\put(.1,2.95){\small $y$}
\put(6.3,3.2){\footnotesize $\Lambda$}
\put(4.74,2.64){\footnotesize $\Gamma$}
\put(4.07,2.39){\footnotesize $\Delta_0$}
\put(3.145,1.86){\footnotesize $\Delta_1$}
\put(6.43,.87){\tiny $\theta\!=\!0$}
\put(3.82,3.78){\tiny $\theta\!=\!\frac{\pi}{2}$}
\put(1.38,3.78){\tiny $\theta\!=\!\theta_\Lambda$}
\end{picture}
\caption{
A sketch of the triangle $\Delta_0$ \eqref{eq:Delta0} and its image $\Delta_1$ under $g$.
\label{fig:DeltaSchem}
} 
\end{center}
\end{figure}

To state the result we first define some geometric structures.
Let $\Gamma$ be the line segment connecting $(1,0)$ and $(0,1)$, i.e.
\begin{equation}
\Gamma = \left\{ (x,1-x) ~\middle|~ 0 \le x \le 1 \right\}.
\label{eq:Gamma}
\end{equation}
Let $\Delta_0$ be the filled triangle with vertices $(0,0)$, $(1,0)$ and $(0,1)$, i.e.
\begin{equation}
\Delta_0 = \left\{ (x,y) ~\middle|~ x \ge 0,\, y \ge 0,\, y \le 1-x \right\}.
\label{eq:Delta0}
\end{equation}
Let
\begin{equation}
\Delta_n = g^n \!\left( \Delta_0 \right), \quad {\rm for~all~} n \ge 1 \;,
\label{eq:Deltai}
\end{equation}
and
\begin{equation}
\Omega_n = \bigcup_{i=0}^n \Delta_i \;, \quad {\rm for~all~} n \ge 0 \;,
\label{eq:Omegaj}
\end{equation}
see Fig.~\ref{fig:DeltaSchem}.

\begin{theorem}
Suppose $\delta_L > 0$, $\delta_R < 0$ and $\tau_L < 2 \sqrt{\delta_L}$
for the map $g$ given by \eqref{eq:g2d}.
Then $\b0$ is an asymptotically stable fixed point of $g$ if and only if
there exist $m,k \in \mathbb{Z}$ such that
$g(\Omega_m) \subset \Omega_m$ and $g^k(\Omega_m) \cap \Gamma = \varnothing$.
\label{th:Ga92}
\end{theorem}

Theorem \ref{th:Ga92} is proved in Appendix \ref{app:proofs} by analysing the dynamics within $\Lambda$.
The condition $\tau_L < 2 \sqrt{\delta_L}$ ensures that $\Lambda$ is globally attracting (Lemma \ref{le:Lambda}).
If $\tau_L \ge 2 \sqrt{\delta_L}$ and $\Lambda$ is forward invariant,
then Theorem \ref{th:Ga92} can be shown to be true
if we only consider orbits in $\Lambda$\removableFootnote{
We do not state this formally as proving it requires generalising
Lemma \ref{le:clh} to invariant sectors which I feel is distracting,
and because we wouldn't end up using it anyway.
In discussing measure stability it is most useful to know that $\Lambda$ is simply forward invariant.
}.

\section{Numerical results for two-dimensional non-invertible maps}
\label{sec:numerics}

\begin{figure}[b!]
\begin{center}
\setlength{\unitlength}{1cm}
\begin{picture}(15.6,7.8)
\put(0,0){\includegraphics[width=15.6cm]{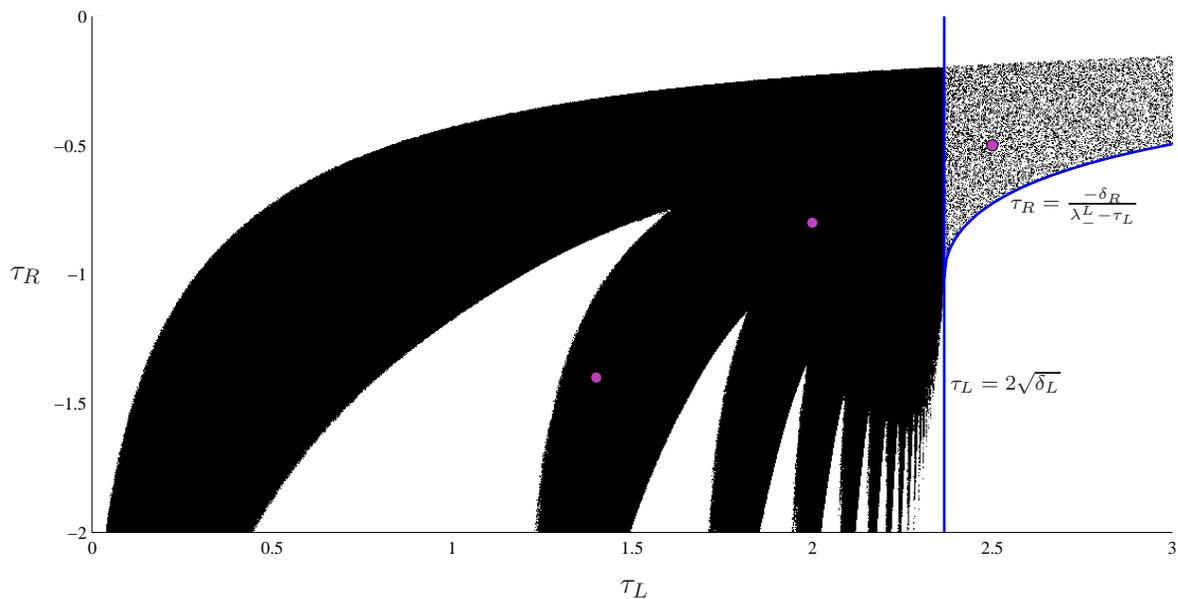}}
\put(8.11,0){\small $\tau_L$}
\put(0,4.17){\small $\tau_R$}
\put(12.5,2.7){\scriptsize $\tau_L = 2 \sqrt{\delta_L}$}
\put(13.3,5.1){\scriptsize $\tau_R = \frac{-\delta_R}{\lambda^L_- - \tau_L}$}
\end{picture}
\caption{
An illustration of the measure stability of $\b0$ for $g$, given by \eqref{eq:g2d}, with \eqref{eq:deltaLR2}.
This figure was obtained by numerically computing the forward orbit of a random point
for a $1024 \times 512$ grid of values of $\tau_L$ and $\tau_R$
and shading each grid point black for which the orbit appeared to converge to $\b0$.
The three coloured circles indicate the parameter values of
Figs.~\ref{fig:ppPolar1}, \ref{fig:ppPolar2}, and \ref{fig:ppPolar3}.
\label{fig:boundednessRegion}
} 
\end{center}
\end{figure}

The results here are motivated by a desire to understand attractors created in border-collision bifurcations
for which both pieces of the corresponding piecewise-linear approximation are area-expanding.
Thus we are particularly interested in the stability of $\b0$ for \eqref{eq:g2d} with $\delta_L > 1$ and $\delta_R < -1$.
For the purposes of illustration, throughout this section we use the fixed values
\begin{equation}
\delta_L = 1.4 \;, \qquad \delta_R = -1.2 \;.
\label{eq:deltaLR2}
\end{equation}
Other values of $\delta_L > 1$ and $\delta_R < -1$ have been found to give similar bifurcation structures.

Fig.~\ref{fig:boundednessRegion} shows values of $\tau_L$ and $\tau_R$ for which the forward orbit
of a random point appeared to converge to $\b0$ through numerical simulation.
This indicates that in the solid black region $\b0$ is measure-$1$ stable.
Numerical results suggest that $G$ has a unique physical measure $\mu$
for $\tau_L < 2 \sqrt{\delta_L}$\removableFootnote{
Otherwise some parts of the boundary of the measure-$1$ stability region would presumably be speckled.
},
and that $\lambda(\mu) = 0$ on the boundary of the region of measure-$1$ stability.

In the speckled region of Fig.~\ref{fig:boundednessRegion},
$\b0$ is measure-$\rho$ stable where $\rho \in (0,1)$ varies slightly throughout this region.
As indicated in Fig.~\ref{fig:ppPolar1},
here forward orbits either approach the invariant ray $\gamma^L_+$ and diverge,
or become trapped in the sector $\Lambda$ and converge to $\b0$.
The speckled region is bounded on the left by the line $\tau_L = 2 \sqrt{\delta_L}$
(to the left of this line $A_L$ has complex-valued eigenvalues and $\gamma^L_+$ does not exist)
and bounded below by the curve $\tau_R = \frac{-\delta_R}{\lambda^L_- - \tau_L}$
(below this curve 
almost all forward orbits approach $\gamma^L_+$ and diverge).

To compute the value of $\rho$ in the speckled region,
let $\psi$ be the unique angle in $\left( \frac{3 \pi}{2}, 2 \pi \right)$ for which $\theta^L_- = G(\psi)$.
For any initial point $(x,y) \in \mathbb{R}^2$ with angle $\theta \in [0,2 \pi)$,
if $\theta \in \left( \theta^L_-,\psi \right)$ then $g^i(x,y)$ approaches $\gamma^L_+$ and diverges,
while if $\theta \in \left[ 0, \theta^L_- \right) \cup \left( \psi, 2 \pi \right)$
then $g^i(x,y)$ becomes trapped in $\Lambda$ and converges to $\b0$.
Therefore the fraction of points whose forward orbit converges to $\b0$ is
\begin{equation}
\rho = 1 - \frac{\psi - \theta^L_-}{2 \pi} \;.
\label{eq:rho3}
\end{equation}
Using \eqref{eq:G2d} and \eqref{eq:thetaJpm}, after simplification
\begin{equation}
\tan \!\left( \psi - \theta^L_- \right) =
\frac{\delta_L - \delta_R + (\tau_L - \tau_R) \left( \lambda^L_- - \tau_L \right)}
{\delta_L \tau_R + (1 - \delta_R + \tau_L \tau_R) \left( \lambda^L_- - \tau_L \right)} \;.
\nonumber
\end{equation}
By applying this formula to the parameter values of Fig.~\ref{fig:ppPolar1},
we obtain $\rho = 0.37$ to two decimal places.

\begin{figure}[b!]
\begin{center}
\setlength{\unitlength}{1cm}
\begin{picture}(15.6,7.8)
\put(0,0){\includegraphics[width=15.6cm]{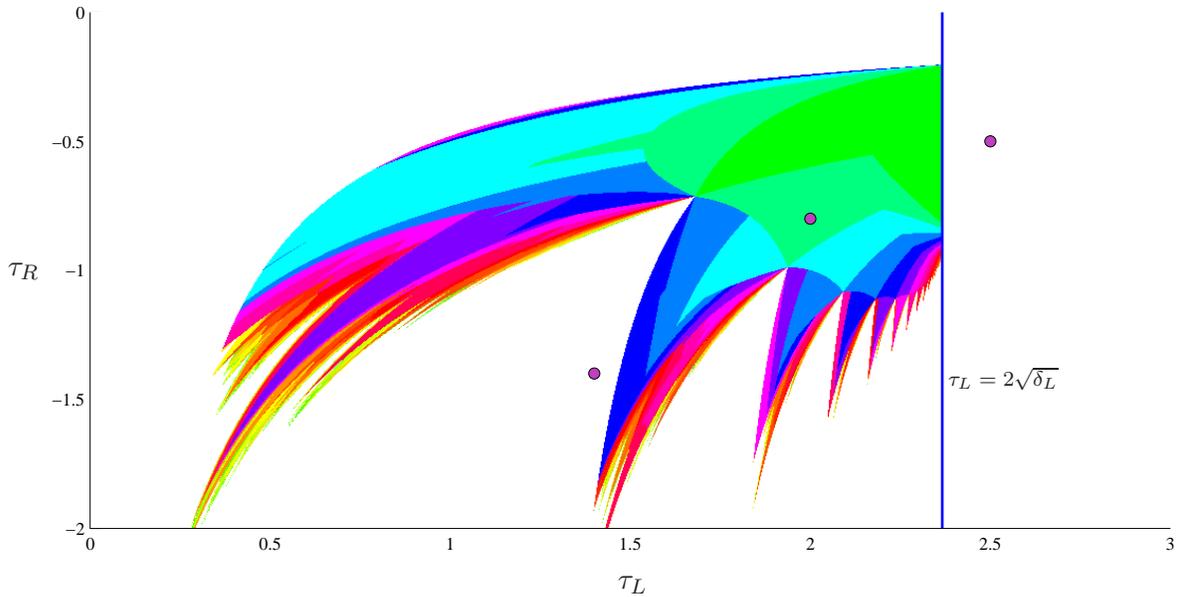}}
\put(8.11,0){\small $\tau_L$}
\put(0,4.17){\small $\tau_R$}
\put(12.5,2.7){\scriptsize $\tau_L = 2 \sqrt{\delta_L}$}
\end{picture}
\caption{
\label{fig:stabilityRegion}
The region of asymptotic stability of $\b0$ for $g$, given by \eqref{eq:g2d}, with \eqref{eq:deltaLR2}.
The region is coloured by the smallest value of $m$ for which
$g(\Omega_m) \subset \Omega_m$, up to $m = 30$, on a $1024 \times 512$ grid of values of $\tau_L$ and $\tau_R$
(green: $m = 1$).
The three coloured circles indicate the parameter values of
Figs.~\ref{fig:ppPolar1}, \ref{fig:ppPolar2}, and \ref{fig:ppPolar3}.
} 
\end{center}
\end{figure}

To identify regions of parameter space where $\b0$ is asymptotically stable, we use Theorem \ref{th:Ga92}.
Numerically we can accurately iterate the entire set $\Delta_0$ under $g$
because it is a triangle and $g$ maps polygons to polygons.
Thus each $\Delta_n$, and hence also each $\Omega_n$, can be encoded with a finite set of points.
Fig.~\ref{fig:stabilityRegion} shows the result
of a numerical search for the smallest value of $m$ for which $g(\Omega_m) \subset \Omega_m$.
Regions are coloured by the value of $m$.
The existence of $k \in \mathbb{Z}$ for which $g^k(\Omega_m) \cap \Gamma = \varnothing$ was not checked
as it is only expected to be false on the boundary of the stability region
where $\b0$ may be Lyapunov stable but not asymptotically stable.
Notice that the asymptotic stability region of Fig.~\ref{fig:stabilityRegion}
is contained within the measure-$1$ stability region shown in Fig.~\ref{fig:boundednessRegion}.
This is because asymptotic stability implies measure-$1$ stability.

By Theorem \ref{th:dilation}, if $\lambda(\mu) < 0$ for every ergodic
invariant probability measure $\mu$ of $G$ then $\b0$ is asymptotically stable,
while if $\lambda(\mu) > 0$ for some such $\mu$ then $\b0$ is not asymptotically stable.
We therefore expect each smooth part of the boundary of the asymptotic stability region
to be where $\lambda(\mu) = 0$ for a particular ergodic measure $\mu$ that
varies smoothly with respect to $\tau_L$ and $\tau_R$.
The complicated nature of the boundary of the asymptotic stability region 
may be explained by the observation that $G$ typically has many
(possibly uncountably many) ergodic invariant probability measures.

\begin{figure}[t!]
\begin{center}
\setlength{\unitlength}{1cm}
\begin{picture}(15.6,8)
\put(0,0){\includegraphics[width=8cm]{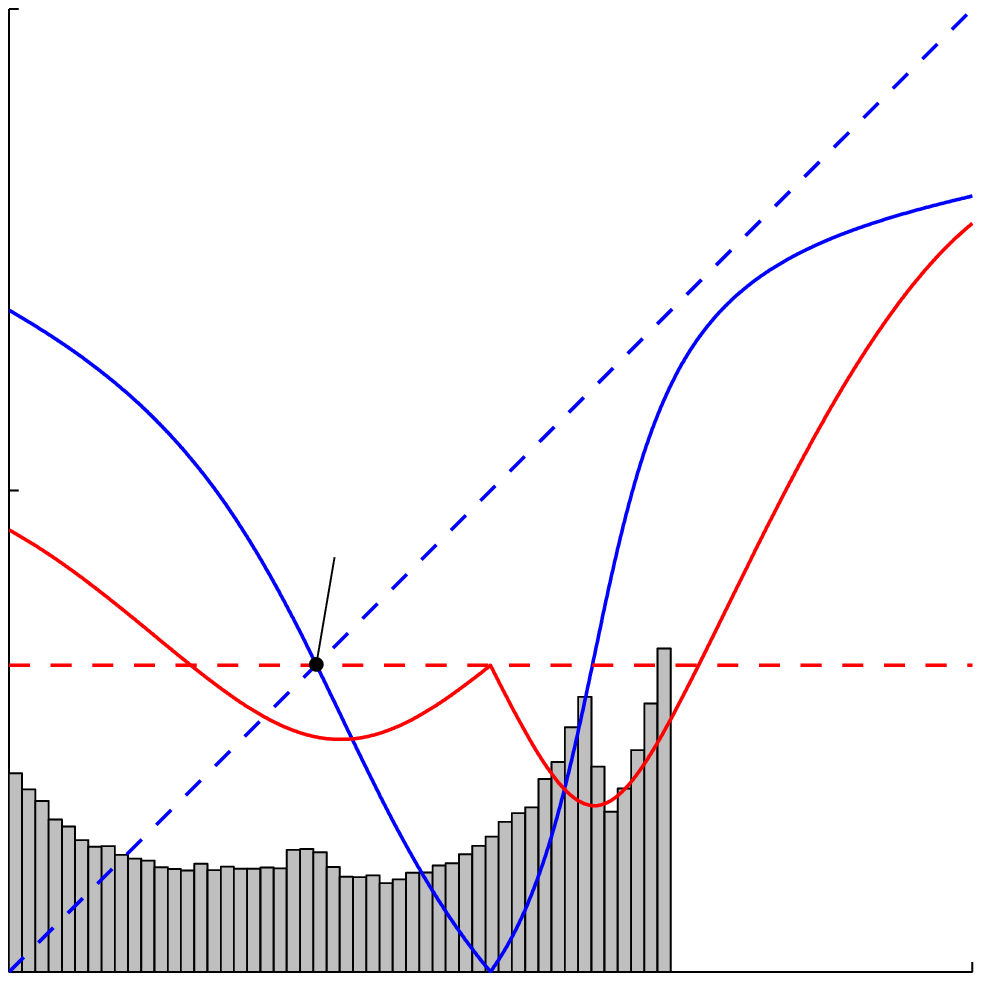}}
\put(8.5,.3){\includegraphics[width=7cm]{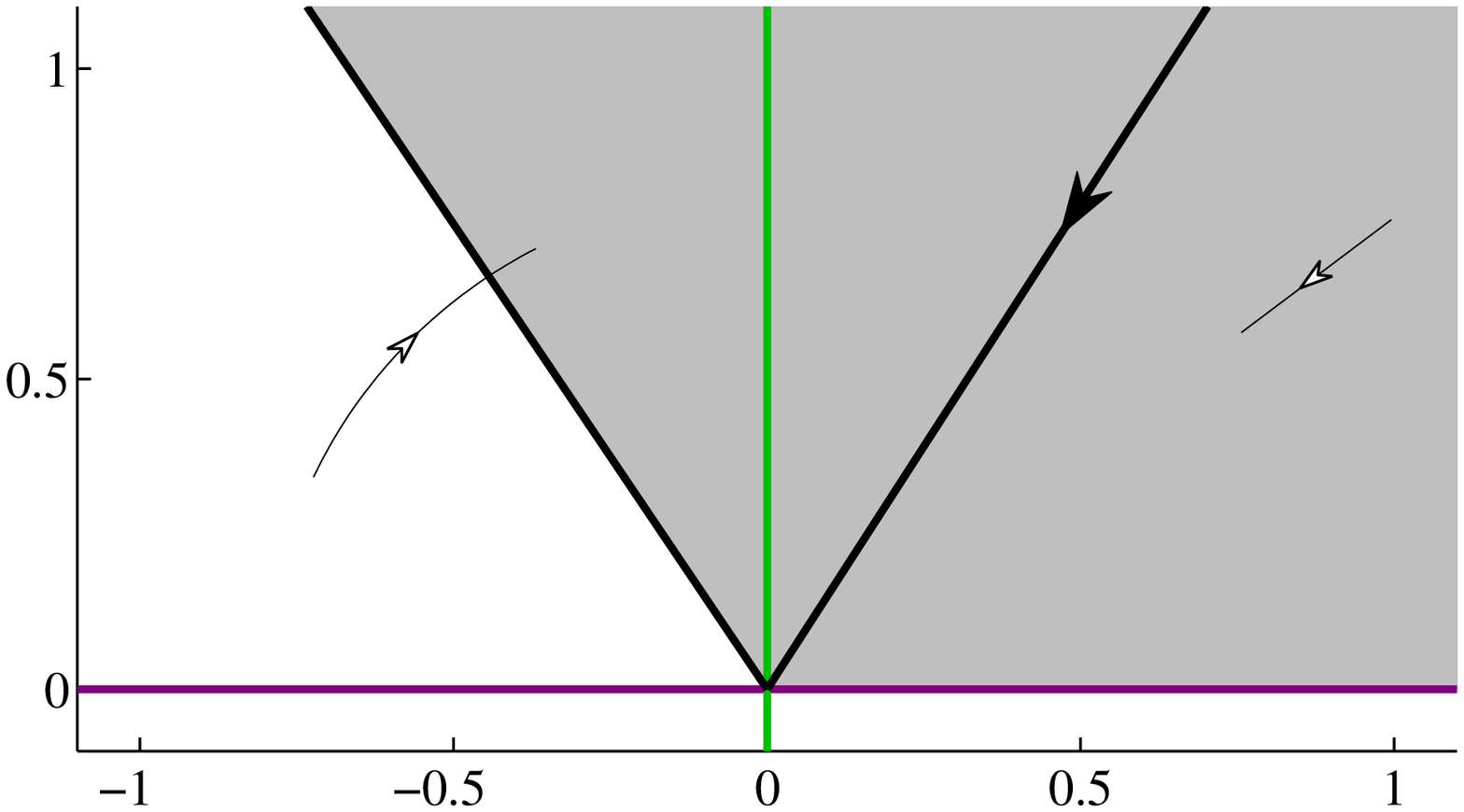}}
\put(1.14,7.85){\sf \bfseries A}
\put(9.41,4.4){\sf \bfseries B}
\put(5.4,.32){\small $\theta$}
\put(.73,.54){\scriptsize $0$}
\put(4.23,.54){\scriptsize $\frac{\pi}{2}$}
\put(7.73,.54){\scriptsize $\pi$}
\put(.58,.74){\scriptsize $0$}
\put(.52,4.26){\scriptsize $\frac{\pi}{2}$}
\put(.58,7.79){\scriptsize $\pi$}
\put(3.1,3.91){\scriptsize $\theta^R_+$}
\put(1.33,3.79){\scriptsize $D(\theta)$}
\put(1.33,5.38){\scriptsize $G(\theta)$}
\put(7,3.15){\tiny $D\!=\!1$}
\put(6.84,7.47){\tiny $G\!=\!\theta$}
\put(12.14,.26){\small $x$}
\put(8.6,3.25){\small $y$}
\put(14.83,1.18){\tiny $\theta\!=\!0$}
\put(12.32,4.06){\tiny $\theta\!=\!\frac{\pi}{2}$}
\put(10.35,4.06){\tiny $\theta\!=\!\theta_\Lambda$}
\put(9.23,1.18){\tiny $\theta\!=\!\pi$}
\put(14.5,3.07){\footnotesize $\Lambda$}
\put(14.2,3.88){\scriptsize $\gamma^R_+$}
\end{picture}
\caption{
The function $D$ \eqref{eq:D2d} and the map $G$ \eqref{eq:G2d} (panel A)
and an illustration of the dynamics of $g$ \eqref{eq:g2d} with
$(\tau_L,\delta_L,\tau_R,\delta_R) = (2,1.4,-0.8,-1.2)$
using the same conventions as Fig.~\ref{fig:ppPolar1}.
\label{fig:ppPolar2}
} 
\end{center}
\end{figure}

Figs.~\ref{fig:ppPolar1}, \ref{fig:ppPolar2}, and \ref{fig:ppPolar3}
illustrate the dynamics of $g$ for three representative combinations of parameter values.
As discussed above, for Fig.~\ref{fig:ppPolar1}, $\b0$ is measure-$\rho$ stable with $\rho \approx 0.37$
(and hence not asymptotically stable).
For Figs.~\ref{fig:ppPolar2} and \ref{fig:ppPolar3}, $\b0$ is measure-$1$ stable.
Here $\Lambda$ is globally attracting and $G$ appears to have a unique physical measure $\mu$
with $\lambda(\mu) = -0.16$ in Fig.~\ref{fig:ppPolar2},
and $\lambda(\mu) = -0.06$ in Fig.~\ref{fig:ppPolar3}, to two decimal places.
However, $\b0$ is asymptotically stable for Fig.~\ref{fig:ppPolar2}, but not for Fig.~\ref{fig:ppPolar3}.
This is because for Fig.~\ref{fig:ppPolar3}, $G$ has a period-three solution
$\{ \theta_0, \theta_1, \theta_2 \}$ with $\lambda(\mu) = 0.03$ (to two decimal places)
for the corresponding invariant probability measure
$\mu = \frac{1}{3} \hspace{-.4mm}\left( \delta_{\theta_0} + \delta_{\theta_1} + \delta_{\theta_2} \right)$.

\begin{figure}[t!]
\begin{center}
\setlength{\unitlength}{1cm}
\begin{picture}(15.6,8)
\put(0,0){\includegraphics[width=8cm]{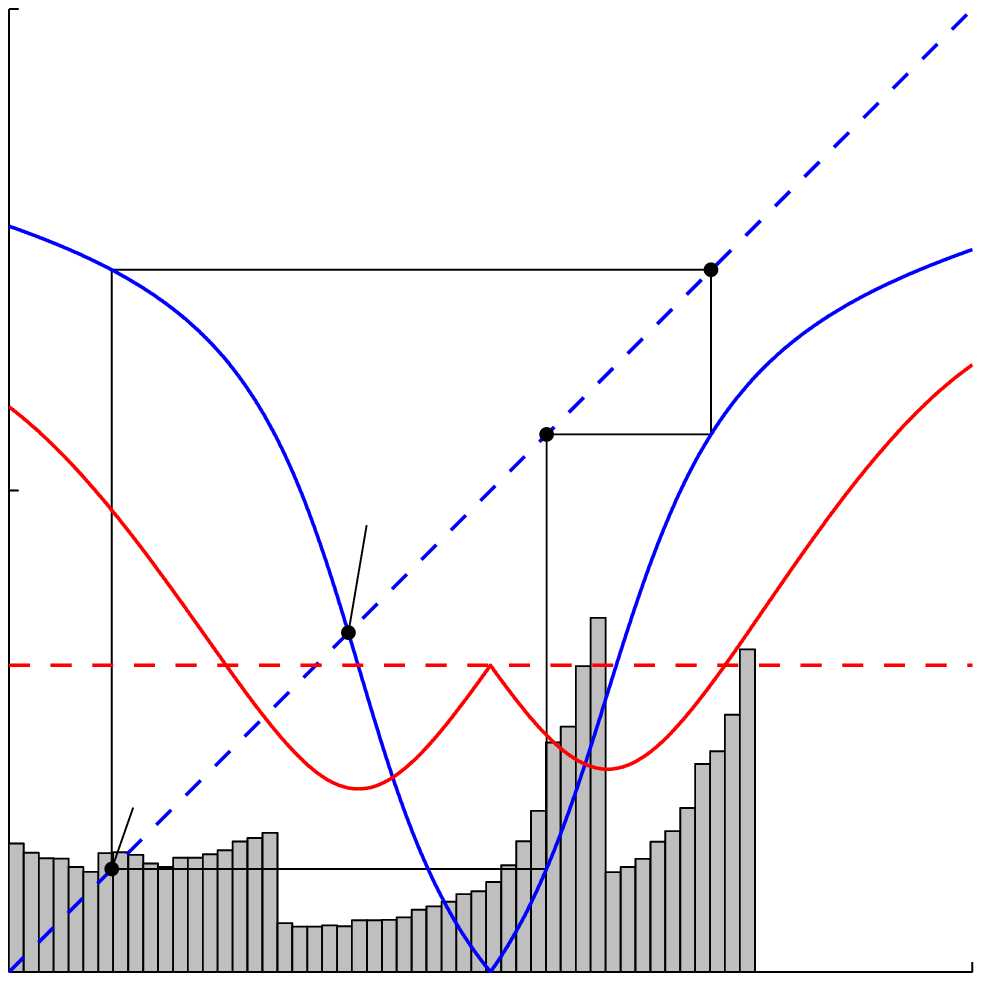}}
\put(8.5,.3){\includegraphics[width=7cm]{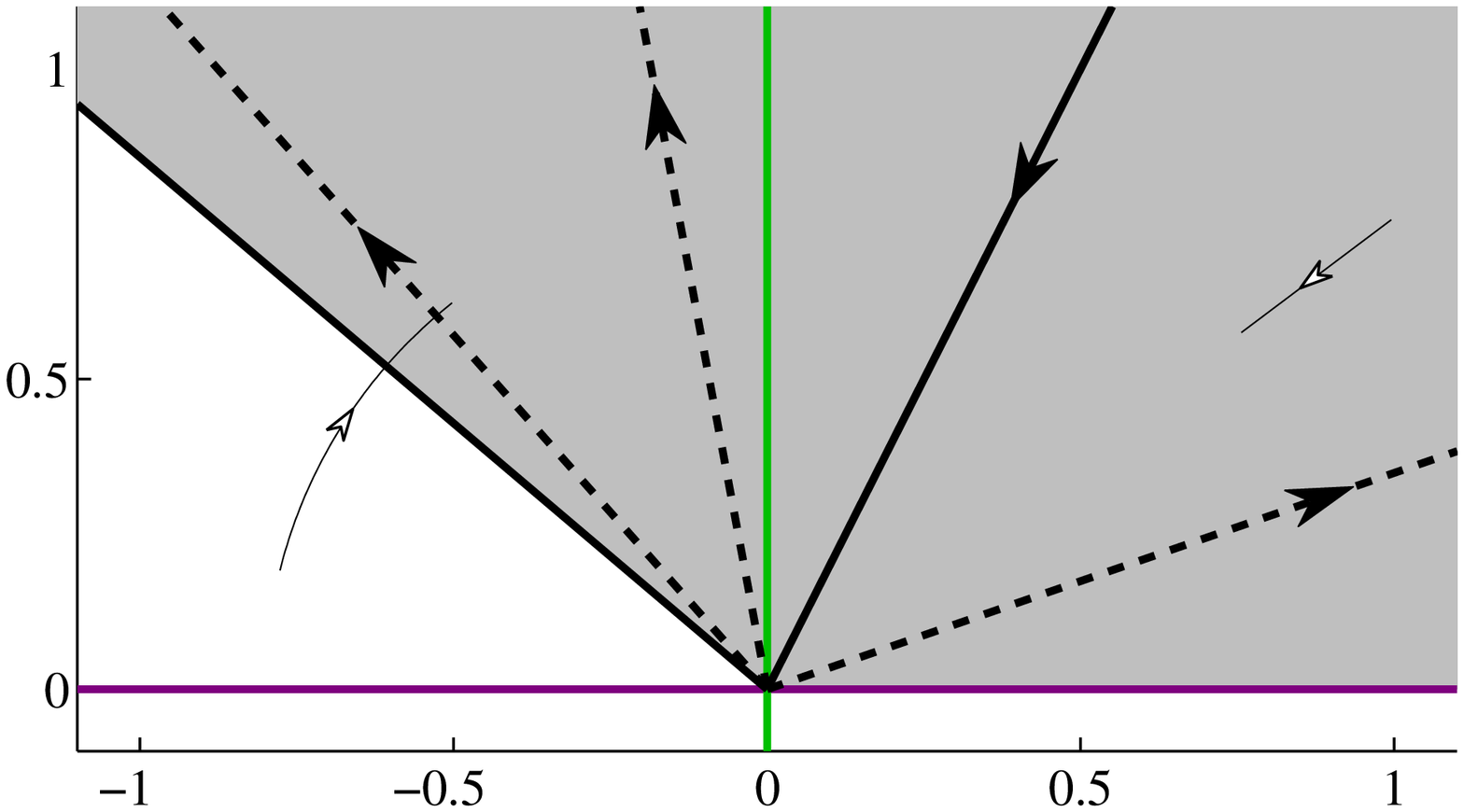}}
\put(1.14,7.85){\sf \bfseries A}
\put(9.41,4.4){\sf \bfseries B}
\put(5.4,.32){\small $\theta$}
\put(.73,.54){\scriptsize $0$}
\put(4.23,.54){\scriptsize $\frac{\pi}{2}$}
\put(7.73,.54){\scriptsize $\pi$}
\put(.58,.74){\scriptsize $0$}
\put(.52,4.26){\scriptsize $\frac{\pi}{2}$}
\put(.58,7.79){\scriptsize $\pi$}
\put(3.33,4.14){\scriptsize $\theta^R_+$}
\put(1.65,2.08){\scriptsize $\theta_0$}
\put(5.67,6.08){\scriptsize $\theta_1$}
\put(4.47,4.87){\scriptsize $\theta_2$}
\put(1.95,3.8){\scriptsize $D(\theta)$}
\put(2.45,5.3){\scriptsize $G(\theta)$}
\put(7,3.15){\tiny $D\!=\!1$}
\put(6.84,7.47){\tiny $G\!=\!\theta$}
\put(12.14,.26){\small $x$}
\put(8.6,3.25){\small $y$}
\put(14.83,1.18){\tiny $\theta\!=\!0$}
\put(12.32,4.06){\tiny $\theta\!=\!\frac{\pi}{2}$}
\put(9.23,3.01){\tiny $\theta\!=\!\theta_\Lambda$}
\put(9.23,1.18){\tiny $\theta\!=\!\pi$}
\put(14.5,3.07){\footnotesize $\Lambda$}
\put(13.79,3.88){\scriptsize $\gamma^R_+$}
\put(15.06,2.3){\scriptsize $\gamma_0$}
\put(9.8,3.99){\scriptsize $\gamma_1$}
\put(11.37,3.99){\scriptsize $\gamma_2$}
\end{picture}
\caption{
The function $D$ \eqref{eq:D2d} and the map $G$ \eqref{eq:G2d} (panel A)
and an illustration of the dynamics of $g$ \eqref{eq:g2d} with
$(\tau_L,\delta_L,\tau_R,\delta_R) = (1.4,1.4,-1.4,-1.2)$
using the same conventions as Fig.~\ref{fig:ppPolar1}.
\label{fig:ppPolar3}
} 
\end{center}
\end{figure}

\section{Summary and outlook for future studies}
\label{sec:conc}

This paper contains several new results regarding the stability of a fixed point
on a switching manifold of a piecewise-smooth continuous map $f$.
For simplicity we always take the fixed point to be the origin $\b0$.
The dynamics near $\b0$ are well approximated by a piecewise-linear map $g$, given by \eqref{eq:g},
involving the Jacobians of each smooth piece of the map evaluated at $\b0$.
If $\b0$ is either asymptotically stable or not Lyapunov stable for $g$,
then, by Theorem \ref{th:hots} and Conjecture \ref{cj:hots2} (if true),
$\b0$ is similarly either asymptotically stable or not Lyapunov stable for $f$.

Any continuous, linearly homogeneous map (such as the piecewise-linear approximation \eqref{eq:g})
can be written as $g(\alpha \bz) = \alpha D(\bz) G(\bz)$, where $\alpha \ge 0$, $\bz \in \mathbb{S}^{d-1}$
and $D$ and $G$ are given by \eqref{eq:DG}.
If $g(\bx) = \b0$ only for $\bx = \b0$, then $G$ is a continuous map on $\mathbb{S}^{d-1}$.
If $g(\bx) = \b0$ for some $\bx \ne \b0$ (not explored in this paper),
then $G$ may have jump discontinuities.

If $\bz$ belongs to the basin of an invariant probability measure $\mu$ of $G$,
then $g^n(\bz) \to \b0$ as $n \to \infty$ if $\lambda(\mu) < 0$,
and $\left\| g^n(\bz) \right\| \to \infty$ as $n \to \infty$ if $\lambda(\mu) > 0$,
where $\lambda(\mu)$ is the average value of $\ln(D)$ over $\mu$, \eqref{eq:lambda}.
By Theorem \ref{th:physical}, if $\lambda(\mu) < 0$ for every physical measure $\mu$ of $G$,
then $\b0$ is a measure-$1$ stable fixed point of $g$.
By Theorem \ref{th:dilation}, if $\lambda(\mu) < 0$ for every ergodic invariant probability measure $\mu$ of $G$,
then $\b0$ is an asymptotically stable fixed point of $g$.

Numerically, the measure stability of $\b0$ can be ascertained by simply
computing the forward orbits of random points near $\b0$, see Fig.~\ref{fig:boundednessRegion}.
A determination of asymptotic stability is considerably more difficult.
It is not uncommon for $G$ to have uncountably many ergodic invariant probability measures,
in which case there seems to be no hope of evaluating $\lambda(\mu)$ for each such measure $\mu$.
For this reason we determined asymptotic stability (shown in Fig.~\ref{fig:stabilityRegion})
by using Theorem \ref{th:Ga92} that applies only to the two-dimensional border-collision normal form.

The stable and unstable manifolds of $\b0$ for a continuous, linearly homogeneous map $g$ may
involve a complicated collection of sectors of $\mathbb{R}^d$ corresponding to invariant probability measures $\mu$ of $G$
for which $\lambda(\mu) < 0$ and $\lambda(\mu) > 0$, respectively.
There remains a critical need to develop a theory for the persistence of these manifolds under $\co(\bx)$ perturbations
as this should enable us to resolve Conjecture \ref{cj:hots2}.

For the two-dimensional border-collision normal form, our numerical results reveal that $\b0$ can be asymptotically stable
even if both smooth components of the map are area-expanding.
Throughout the stability region shown in Fig.~\ref{fig:stabilityRegion},
$A_R$ has an eigenvalue with modulus less than $1$.
We conjecture that, for the general piecewise-linear map \eqref{eq:g},
$\b0$ cannot be asymptotically stable if all eigenvalues of $A_L$ and $A_R$ have modulus greater than $1$.
If true, this result may be rather difficult to prove because, as we have seen,
the stability of $\b0$ is more closely related to invariant probability measures of $G$
than the eigenvalues of $A_L$ and $A_R$.

It remains to show how the results of this paper can be used to establish
conditions for the existence of attractors created in border-collision bifurcations.
The key idea is that the stability of the fixed point at the bifurcation
should imply the existence of a structurally stable attractor.
In cases for which both components of the corresponding piecewise-linear approximation are area-expanding,
it may be possible to prove that such attractors are robustly chaotic.
For the border-collision normal form, robust chaos is described in \cite{BaYo98}
and the existence of multi-dimensional chaotic attractors
has recently been demonstrated by using results for general piecewise-expanding maps \cite{Gl15b,Gl16e}.
It also remains to systematically study the stability of fixed points on
switching manifolds for piecewise-smooth continuous maps for which one-sided derivatives are not locally bounded,
such as maps with a square-root singularity that arise as return maps for regular grazing bifurcations \cite{DiBu01}.

\appendix
\section{Additional proofs}
\label{app:proofs}

\begin{proof}[Proof of Lemma \ref{le:clh}]
We prove the two parts of the lemma in order.
\begin{enumerate}
\item
By \eqref{eq:linearlyHomogeneous}, $g^n(\bx) \to \b0$ as $n \to \infty$ for all $\bx \in \mathbb{R}^d$.
Thus it remains to show that $\b0$ is a Lyapunov stable fixed point of $g$.

Suppose for a contradiction that $\b0$ is not Lyapunov stable.
By \eqref{eq:linearlyHomogeneous} this means that
for all $k \ge 1$ there exists $\bx_k \in \overline{B}_1$ and $n_k \in \mathbb{Z}$
with $\left\| g^{n_k}(\bx_k) \right\| \ge k$.
By \eqref{eq:linearlyHomogeneous} we can require $\| \bx_k \| = 1$ for each $k$.
We can also assume
\begin{equation}
\left\| g^n(\bx_k) \right\| \ge 1 \;, \quad
{\rm for~all~} n = 0,1,\ldots,n_k \;,
\label{eq:gixmLarge}
\end{equation}
because if \eqref{eq:gixmLarge} does not hold then we can replace $\bx_k$
with $\frac{g^n(\bx_k)}{\left\| g^n(\bx_k) \right\|}$
for the largest $n < n_k$ for which $\left\| g^n(\bx_k) \right\| < 1$
(and replace $n_k$ with $n_k - n$).
Note that we necessarily have $n_k \to \infty$ as $k \to \infty$.

Since $\mathbb{S}^{d-1}$ is compact, $\{ \bx_k \}$ has a convergent subsequence.
That is $\bx_{k_j} \to \by$ as $j \to \infty$ for some $k_j \in \mathbb{Z}$ and $\by \in \mathbb{S}^{d-1}$.
Next we show that $g^n(\by) \not\to \b0$ as $n \to \infty$.

Choose any $N \ge 1$.
Since $g^N$ is continuous there exists $\delta > 0$ such that
$\left\| g^N(\bx) - g^N(\by) \right\| < \frac{1}{2}$ whenever $\| \bx - \by \| < \delta$.
Since $\bx_{k_j} \to \by$ and $n_{k_j} \to \infty$ as $j \to \infty$,
there exists $j \in \mathbb{Z}$ such that
$\left\| \bx_{k_j} - \by \right\| < \delta$ and $n_{k_j} \ge N$.
Then $\left\| g^N \!\left( \bx_{k_j} \right) - g^N(\by) \right\| < \frac{1}{2}$
and $\left\| g^N \!\left( \bx_{k_j} \right) \right\| \ge 1$ by \eqref{eq:gixmLarge}.
Thus
\begin{equation}
\left\| g^N(\by) \right\| \ge \left\| g^N \!\left( \bx_{k_j} \right) \right\| -
\left\| g^N \!\left( \bx_{k_j} \right) - g^N(\by) \right\| > 1 - \frac{1}{2} = \frac{1}{2} \;.
\nonumber
\end{equation}
Hence $g^n(\by) \not\to \b0$ as $n \to \infty$, which is a contradiction.
\item
Choose any $\ee > 0$.
We have just shown that $\b0$ is a Lyapunov stable fixed point of $g$,
hence there exists $\delta_1 > 0$ such that
\begin{equation}
g^n(\bx) \in B_{\frac{\ee}{2}} \;, \quad {\rm for~all~} \bx \in B_{\delta_1} {\rm ~and~all~} n \ge 0 \;.
\label{eq:gLyapunovStableTilde}
\end{equation}
Choose any $\by \in \overline{B}_r$.
Since $g^n(\by) \to \b0$ as $n \to \infty$ there exists $N_1 \in \mathbb{Z}$ such that
\begin{equation}
g^n(\by) \in B_{\frac{\delta_1}{2}} \;, \quad {\rm for~all~} n \ge N_1 \;.
\label{eq:gAsyStableTilde}
\end{equation}
Since $g$ is continuous, there exists $\delta > 0$ such that
\begin{align}
\left\| g^n(\bx) - g^n(\by) \right\| < \ee \;, \quad
& {\rm for~all~} \bx \in \overline{B}_r {\rm ~with~} \| \bx - \by \| < \delta \nonumber \\
& {\rm and~all~} n = 0,1,\ldots,N_1 \;,
\label{eq:gContinuousTilde1}
\end{align}
and
\begin{equation}
\left\| g^{N_1}(\bx) - g^{N_1}(\by) \right\| < \frac{\delta_1}{2} \;, \quad
{\rm for~all~} \bx \in \overline{B}_r {\rm ~with~} \| \bx - \by \| < \delta \;.
\label{eq:gContinuousTilde2}
\end{equation}
By \eqref{eq:gAsyStableTilde} and \eqref{eq:gContinuousTilde2},
for any $\bx \in \overline{B}_r$ with $\| \bx - \by \| < \delta$, we have
\begin{equation}
\left\| g^{N_1}(\bx) \right\| \le
\left\| g^{N_1}(\bx) - g^{N_1}(\by) \right\| + \left\| g^{N_1}(\by) \right\| <
\frac{\delta_1}{2} + \frac{\delta_1}{2} = \delta_1 \;.
\nonumber
\end{equation}
That is, $g^{N_1}(\bx) \in B_{\delta_1}$,
and so by \eqref{eq:gLyapunovStableTilde}
we have $g^n(\bx) \in B_{\frac{\ee}{2}}$ for all $n \ge N_1$.
This shows that the sequence of functions $\{ g^n \}$ is uniformly bounded on $\overline{B}_r$.
Also $g^n(\by) \in B_{\frac{\ee}{2}}$ for all $n \ge N_1$
by \eqref{eq:gLyapunovStableTilde} and \eqref{eq:gAsyStableTilde}.
Therefore for all $\bx \in \overline{B}_r$ with $\| \bx - \by \| < \delta$ and all $n \ge N_1$ we have
\begin{equation}
\left\| g^n(\bx) - g^n(\by) \right\| \le
\left\| g^n(\bx) \right\| + \left\| g^n(\by) \right\| \le
\frac{\ee}{2} + \frac{\ee}{2} = \ee \;.
\label{eq:equicontinuousLargei}
\end{equation}
By \eqref{eq:gContinuousTilde1} and \eqref{eq:equicontinuousLargei}
we can conclude that $\{ g^n \}$ is equicontinuous on $\overline{B}_r$.

Since $\overline{B}_r$ is compact and $\{ g^n \}$ is uniformly bounded and equicontinuous,
by the Arzel\`{a}-Ascoli theorem there exists a subsequence of $\{ g^n \}$
that converges uniformly \cite{RoFi10,So03}\removableFootnote{
This part of the Arzel\`{a}-Ascoli theorem is quite easy to prove:
use a diagonalisation argument and a ``$\frac{\ee}{3}$'' argument.
In the special case that the $g^n$ converge, the diagonisation argument is not needed.
In the extra special case that the $g^n$ converge to a constant, as we have here,
an ``$\frac{\ee}{2}$'' argument suffices, as given in the proof of Lemma C.2 of \cite{DiNo08}.

The most common version of the Arzel\`{a}-Ascoli theorem is quite different to what we are using:
$\{ g^n \}$ is compact if and only if it is closed, bounded, and equicontinuous.
}.
Since $\{ g^n \}$ converges to $\b0$ we can conclude that $\{ g^n \}$ converges uniformly to $\b0$.
\end{enumerate}
\end{proof}

\begin{proof}[Proof of Lemma \ref{le:figiBound}]
Let
\begin{equation}
K = \max \left[ 1, \max_{\bz \in \mathbb{S}^{d-1}} \| g(\bz) \| \right],
\nonumber
\end{equation}
which is well-defined because $g$ is continuous and $\mathbb{S}^{d-1}$ is compact.
Since $g$ is linearly homogeneous, for any $\bx \in \mathbb{R}^d$ we have
$\| g(\bx) \| \le K \| \bx \|$, and therefore
\begin{equation}
\left\| g^i(\bx) \right\| \le K^i \| \bx \|, \quad
{\rm for~all~} \bx \in \mathbb{R}^d {\rm ~and~all~} i \ge 0 \;.
\label{eq:giBound}
\end{equation}
The result \eqref{eq:figiBound} is trivial for $n = 1$.
Choose any $n \ge 2$ and any $\ee > 0$.
Let $\eta_1 = \frac{\ee}{2 K^n}$.
By the Heine-Cantor theorem,
since $g$ is continuous on the compact set $\overline{B}_1$ it is uniformly continuous on $\overline{B}_1$.
Therefore we can iteratively find $\eta_2,\eta_3,\ldots,\eta_n > 0$ such that
\begin{align}
\| g(\bx) - g(\by) \| < \frac{\eta_j}{2} \;, \quad
& {\rm for~all~} \bx, \by \in \overline{B}_1 {\rm ~with~} \| \bx - \by \| < \eta_{j+1} \;, \nonumber \\
& {\rm for~all~} j = 1,2,\ldots,n-1 \;,
\label{eq:gUniformContinuity}
\end{align}
and we assume $\eta_1 \ge \eta_2 \ge \cdots \ge \eta_n$.
Since $g$ is linearly homogeneous, \eqref{eq:gUniformContinuity} can be generalised to
\begin{align}
\| g(\bx) - g(\by) \| < \frac{\alpha \eta_j}{2} \;, \quad
& {\rm for~all~} \bx, \by \in \overline{B}_\alpha {\rm ~with~} \| \bx - \by \| < \alpha \eta_{j+1} \;, \nonumber \\
& {\rm for~all~} j = 1,2,\ldots,n-1 \;, {\rm ~and~all~} \alpha > 0 \;.
\label{eq:gUniformContinuity2}
\end{align}
Let $h(\bx) = f(\bx) - g(\bx)$.
Since $h$ is $\co(\bx)$, there exists $\delta_1 > 0$ such that
\begin{equation}
\| h(\bx) \| < \frac{\eta_n}{2} \| \bx \| \;, \quad {\rm for~all~} \bx \in B_{\delta_1} \;.
\label{eq:hLittleO}
\end{equation}

Let $\delta = \frac{\delta_1}{2 K^n}$.
Next we use induction on $i$ to show that
\begin{equation}
\left\| f^i(\bx) - g^i(\bx) \right\| \le 2 K^n \eta_{n-i+1} \| \bx \| \;, \quad
{\rm for~all~} \bx \in B_\delta {\rm ~and~all~} i = 1,2,\ldots,n \;.
\label{eq:figiDiff}
\end{equation}
This will complete the proof because \eqref{eq:figiBound}
follows immediately from \eqref{eq:figiDiff} with $i = n$.

Choose any $\bx \in B_\delta$.
Equation \eqref{eq:figiDiff} is true for $i = 1$ because by \eqref{eq:hLittleO}
\begin{equation}
\| f(\bx) - g(\bx) \| = \| h(\bx) \| \le \frac{\eta_n}{2} \| \bx \| \le 2 K^n \eta_n \| \bx \|.
\nonumber
\end{equation}
Suppose \eqref{eq:figiDiff} is true for some $i = k < n$ (this is the inductive hypothesis).
To verify \eqref{eq:figiDiff} for $i = k+1$ we first use \eqref{eq:giBound} and the inductive hypothesis to obtain
\begin{equation}
\left\| f^k(\bx) \right\| \le
\left\| f^k(\bx) - g^k(\bx) \right\| + \left\| g^k(\bx) \right\| \le
2 K^n \eta_{n-k+1} \| \bx \| + K^k \| \bx \| \le 2 K^n \| \bx \|.
\label{eq:fkBound}
\end{equation}
Since $2 K^n \| \bx \| < \delta_1$, by \eqref{eq:hLittleO} and \eqref{eq:fkBound} we have
\begin{equation}
\left\| h \!\left( f^k(\bx) \right) \right\| \le
\frac{\eta_n}{2} \left\| f^k(\bx) \right\| \le
K^n \eta_n \| \bx \|.
\label{eq:hfkBound}
\end{equation}
Next we use \eqref{eq:gUniformContinuity2} with $j = n-k$,
$f^k(\bx)$ and $g^k(\bx)$ in place of $\bx$ and $\by$,
and $\alpha = 2 K^n \| \bx \|$ (which is justified by \eqref{eq:giBound} and \eqref{eq:fkBound})
and the inductive hypothesis to obtain
\begin{equation}
\left\| g \!\left( f^k(\bx) \right) - g \!\left( g^k(\bx) \right) \right\| \le
\big( 2 K^n \| \bx \| \big) \frac{\eta_{n-k}}{2} = K^n \eta_{n-k} \| \bx \|.
\label{eq:gfkggkDiff}
\end{equation}
We then write $f^{k+1}(\bx) = g \!\left( f^k(\bx) \right) + h \!\left( f^k(\bx) \right)$
and add \eqref{eq:hfkBound} and \eqref{eq:gfkggkDiff} to obtain
\begin{equation}
\left\| f^{k+1}(\bx) - g^{k+1}(\bx) \right\| \le
K^n (\eta_{n-k} + \eta_n) \| \bx \| \le
2 K^n \eta_{n-k} \| \bx \|, \nonumber
\end{equation}
which verifies \eqref{eq:figiDiff} for $i = k+1$.
\end{proof}

\begin{proof}[Proof of Lemma \ref{le:GL}]\removableFootnote{
Here is my original proof of Lemma \ref{le:GL} which uses $G$ directly.

We first note that $G \!\left( \frac{\pi}{2} \right) < \frac{\pi}{2}$, $G(\pi) < \pi$, and
$G$ is increasing on $\left[ \frac{\pi}{2}, \pi \right]$ because
$\frac{dG}{d\theta} = \frac{\delta_L \sec^2(\theta)}
{\left( \tau_L + \tan(\theta) \right)^2 + \delta_L^2}$
is strictly positive.
\begin{enumerate}
\item 
If $\tau_L < 2 \sqrt{\delta_L}$, then $A_L$ has no positive eigenvalues.
Thus $G$ has no fixed points on $\left[ \frac{\pi}{2}, \pi \right]$ by Lemma \ref{le:fps}.
Therefore $G(\theta) < \theta$ on $\left[ \frac{\pi}{2}, \pi \right]$, and so
for all $\theta \in \left( \frac{\pi}{2}, \pi \right)$
there exists $i > 0$ such that $G^i(\theta) \in \left[ 0, \frac{\pi}{2} \right]$.
\item
If $\tau_L > 2 \sqrt{\delta_L}$, then $A_L$ has two distinct positive eigenvalues.
Thus $G$ has exactly two fixed points in $\left( \frac{\pi}{2}, \pi \right)$, namely $\theta^L_+$ and $\theta^L_-$.
We have $\theta^L_- < \tilde{\theta} < \theta^L_+$, where $\tilde{\theta} = \tan^{-1} \left( \frac{-\tau_L}{2} \right)$.
Notice $G(\tilde{\theta}) > \tilde{\theta}$ because
$G(\tilde{\theta}) = \tan^{-1} \!\left( \frac{-2 \delta_L}{\tau_L} \right)$ and $\tau_L > 2 \sqrt{\delta_L}$.
Therefore $G(\theta) > \theta$ on $(\theta^L_-, \theta^L_+)$;
also $G(\theta) < \theta$ on $\left[ \frac{\pi}{2}, \theta^L_- \right) \cup (\theta^L_+, \pi]$.
Recalling also that $G$ is increasing on $\left[ \frac{\pi}{2}, \pi \right]$,
the required conclusions regarding $G^i(\theta)$ follow.
\end{enumerate}
}
If $\tau_L < 2 \sqrt{\delta_L}$ then the eigenvalues of $A_L$ are complex-valued
and the result is an immediate consequence of the observation that iterates of $A_L \bx$ rotate
clockwise about $\b0$.

If $\tau_L > 2 \sqrt{\delta_L}$ then $A_L$ has two positive eigenvalues $\lambda^L_-$ and $\lambda^L_+$.
Since $\delta_L > 0$ the angles $\theta^L_-$ and $\theta^L_+$ belong to $\left( \frac{\pi}{2}, \pi \right)$.
By Lemma \ref{le:fps}, $\theta^L_-$ and $\theta^L_+$ are fixed points of $G$.
Since $\lambda^L_- < \lambda^L_+$, for the map $A_L \bx$
the eigenspace for $\lambda^L_-$ is either repelling or less strongly attracting than the
eigenspace for $\lambda^L_+$.
Thus $\theta^L_-$ and $\theta^L_+$ are repelling and attracting fixed points of $G$ respectively
and the result follows.
\end{proof}

\begin{proof}[Proof of Lemma \ref{le:GR}]\removableFootnote{
Here is my original proof of Lemma \ref{le:GR} which uses $G$ directly.

With $\delta_R < 0$, the matrix $A_R$ has a single positive eigenvalue $\lambda^R_+$.
Thus by Lemma \ref{le:fps}, $G$ has a unique fixed point on $\left[ 0, \frac{\pi}{2} \right]$ given by $\theta^R_+$.

Under the change of variables $s = \tan(\theta)$,
the part of \eqref{eq:G2d} corresponding to $\theta \in \left[ 0, \frac{\pi}{2} \right)$ is transformed to
\begin{equation}
\tilde{G}_R(s) = \frac{-\delta_R}{\tau_R + s} \;, \quad {\rm for~} s \ne -\tau_R \;,
\label{eq:tildeGR}
\end{equation}
for $s \in [0,\infty)$.
We use $\tilde{G}_R$ to complete the proof as it is algebraically simpler than $G$.

The unique fixed point of $\tilde{G}_R$ on $[0,\infty)$ is
$s^R_+ = \tan \!\left( \theta^R_+ \right) = \lambda^R_+ - \tau_R$.
The stability multiplier of $s^R_+$ is readily found to be
\begin{equation}
\frac{d \tilde{G}_R}{d s} \left( s^R_+ \right) =
\frac{\tau_R}{\lambda^R_+} - 1 \;.
\label{eq:dtildeGRds}
\end{equation}

\begin{enumerate}
\item
If $\tau_R > 0$, then $\tilde{G}_R : [0,\infty) \to [0,\infty)$
and so $[0,\infty)$ is forward invariant for $\tilde{G}_R$.
Since $s = \tan(\theta)$ is a homeomorphism for $\theta \in \left[ 0, \frac{\pi}{2} \right)$,
we can conclude that $\left[ 0, \frac{\pi}{2} \right]$ is forward invariant for $G$.

By \eqref{eq:dtildeGRds} we have $-1 < \frac{d \tilde{G}_R}{d s} \left( s^R_+ \right) < 0$,
and so $s^R_+$ is an attracting fixed point.
To show that it attracts all points in $[0,\infty)$ we look at the second iterate $\tilde{G}_R^2$.
The only fixed point of $\tilde{G}_R^2$ is $s^R_+$
because $(\lambda^{R}_+)^2$ is the only positive eigenvalue of $A_R^2$.
Observe $0 < \frac{d}{d s} \tilde{G}_R^2 \left( s^R_+ \right) < 1$
and that $\tilde{G}_R^2$ is increasing on $[0,\infty)$ because
$\frac{d}{d s} \tilde{G}_R^2(s) =
\frac{\delta_R^2}{\left( \tau_R^2 - \delta_R + \tau_R s \right)^2}$
is strictly positive.
Therefore $\tilde{G}_R^i(s) \to s^R_+$ as $i \to \infty$ for all $s \in [0,\infty)$,
and the analogous statement is true for $G$.
\item
If $\tau_R < 0$, then $\frac{d \tilde{G}_R}{d s} \left( s^R_+ \right) < -1$ by \eqref{eq:dtildeGRds}.
Thus $s^R_+$ is a repelling fixed point of $\tilde{G}_R$ and so $\theta^R_+$ is a repelling fixed point of $G$.
\end{enumerate}
}
With $\delta_R < 0$ the eigenvalues of $A_R$ satisfy $\lambda^R_- < 0 < \lambda^R_+$.
Also $\theta^R_- \in \left( \frac{\pi}{2}, \pi \right)$ and $\theta^R_+ \in \left( 0, \frac{\pi}{2} \right)$.
By Lemma \ref{le:fps}, $G$ has the unique fixed point $\theta^R_+$.

If $\tau_R < 0$ then $\lambda^R_+ < \left| \lambda^R_- \right|$ and 
for the map $A_R \bx$ the eigenspace for $\lambda^R_+$ is either repelling or less strongly attracting
than the eigenspace for $\lambda^R_-$.
It follows that $\theta^R_+$ is repelling for $G$.
If $\tau_R > 0$ then $\lambda^R_+ > \left| \lambda^R_- \right|$ and it similarly follows
that $\theta^R_+$ is globally attracting for $G$.
\end{proof}

\begin{proof}[Proof of Lemma \ref{le:Lambda}]
To show that $\Lambda$ is forward invariant it suffices to show that 
$G(\theta) \in [0,\theta_\Lambda]$ for all $\theta \in [0,\theta_\Lambda]$.
We have
\begin{equation}
\frac{dG}{d\theta} = \begin{cases}
\frac{\delta_R \sec^2(\theta)}{\left( \tau_R + \tan(\theta) \right)^2 + \delta_R^2} \;, &
\theta \in \left[ 0, \frac{\pi}{2} \right), \\
\frac{\delta_L \sec^2(\theta)}{\left( \tau_L + \tan(\theta) \right)^2 + \delta_L^2} \;, &
\theta \in \left( \frac{\pi}{2}, \pi \right).
\end{cases}
\label{eq:dGdtheta}
\end{equation}
Thus $G$ is decreasing on $\left[ 0, \frac{\pi}{2} \right)$ and increasing on $\left( \frac{\pi}{2}, \pi \right)$
because $\delta_R < 0$ and $\delta_L > 0$.
Thus the minimum value of $G$ over $[0,\theta_\Lambda]$ is $G \!\left( \frac{\pi}{2} \right) = 0$
which belongs to $[0,\theta_\Lambda]$.
Also the maximum value of $G$ over $[0,\theta_\Lambda]$ is either $G(0) = \theta_\Lambda$,
which belongs to $[0,\theta_\Lambda]$, or $G(\theta_\Lambda)$.
Thus if we can show that $G(\theta_\Lambda) \in [0,\theta_\Lambda]$
then we have verified that $G(\theta) \in [0,\theta_\Lambda]$ for all $\theta \in [0,\theta_\Lambda]$.
But the only scenario in which $G(\theta_\Lambda) \notin [0,\theta_\Lambda]$
is if $\theta_\Lambda \in \left( \frac{\pi}{2}, \pi \right)$, requiring $\tau < 0$,
and $G(\theta_\Lambda) > \theta_\Lambda$, requiring $\tau_L > 2 \sqrt{\delta_L}$
(so that $G$ has fixed points $\theta^L_-, \theta^L_+ \in \left( \frac{\pi}{2}, \pi \right)$)
and $\theta^L_- < \theta_\Lambda < \theta^L_+$.
This verifies the forward invariance claim of $\Lambda$.

For any $(x,y) \in \mathbb{R}^2$, $g(x,y)$ is in the upper half-plane $y \ge 0$.
In the left half-plane $g$ is the linear map $A_L \bx$
and if $\tau_L < 2 \sqrt{\delta_L}$ this corresponds to clockwise rotation about $\b0$.
Thus there exists $M \in \mathbb{Z}$ such that
if $g(x,y)$ is in the left half-plane then $g^i(x,y)$ is in the right half-plane for some $i \le M-1$.
Then $g^{i+1}(x,y) \in \Lambda$ and, by the forward invariance of $\Lambda$,
$g^M(x,y) \in \Lambda$ as required.
\end{proof}

\begin{proof}[Proof of Theorem \ref{th:Ga92}]
First suppose that $\b0$ is an asymptotically stable fixed point of $g$.
Since $\Delta_1$ is a triangle with points on opposite boundaries of $\Lambda$,
there exists $\ee > 0$ such that $B_\ee \cap \Lambda \subset \Delta_1$,
and we assume $\ee \le \frac{1}{\sqrt{2}}$ such that $B_\ee \cap \Gamma = \varnothing$.
Since $\b0$ is Lyapunov stable and $\Lambda$ is uniformly globally attracting (by Lemma \ref{le:Lambda})
there exists $\delta > 0$ such that for all $(x,y) \in B_\delta$
we have $g^n(x,y) \in B_\ee \cap \Lambda$ for all $n \ge 0$.
By Lemma \ref{le:clh}(ii) and linear homogeneity, $g^n(\Delta_0) \to \b0$ uniformly.
Thus there exists $m \in \mathbb{Z}$ such that
$\Delta_{m+1} \subset B_\delta \cap \Lambda$.
Thus $\Delta_{m+1} \subset B_\ee \cap \Lambda \subset \Delta_1$
and so $\Delta_{m+1} \subset \Omega_m$.
Trivially $\Delta_1 \cup \Delta_2 \cup \cdots \cup \Delta_m \subset \Omega_m$,
thus $g(\Omega_m) = \Delta_1 \cup \Delta_2 \cup \cdots \cup \Delta_{m+1} \subset \Omega_m$ as required.

Next we show that $g^k(\Omega_m) \cap \Gamma = \varnothing$ for $k = m+1$.
Since $\Delta_{m+1} \subset B_\delta$, we have $\Delta_{m+1}, \Delta_{m+2}, \ldots \subset B_\ee$.
Thus $g^{m+1}(\Omega_m) = \Delta_{m+1} \cup \Delta_{m+2} \cup \cdots \cup \Delta_{2 m + 1} \subset B_\ee$
and so $g^{m+1}(\Omega_m) \cap \Gamma = \varnothing$.

Conversely suppose there exist $m,k \in \mathbb{Z}$ such that
$g(\Omega_m) \subset \Omega_m$ and $g^k(\Omega_m) \cap \Gamma = \varnothing$.
Let $Q_1 = \{ (x,y) ~|~ x \ge 0, y \ge 0 \}$.
There exists $0 \le \alpha < 1$ such that
$g^k(\Omega_m) \cap Q_1 \subset \alpha \Delta_0 \subset \alpha \Omega_m$
(where by multiplying a set by $\alpha$ we mean that every point in the set is scaled by $\alpha$).
As a simple extension of Lemma \ref{le:Lambda}, there exists $M \in \mathbb{Z}$ such that
for every $(x,y) \in \mathbb{R}^2$ we have $g^i(x,y) \in \Lambda \cap Q_1$ for some $i \le M$.
Then for every $(x,y) \in g^k(\Omega_m)$ we have $g^i(x,y) \in Q_1$ for some $i \le M$.
Thus $g^i(x,y) \in g^k(\Omega_m) \cap Q_1$, since $\Omega_m$ is forward invariant under $g$.
Thus $g^i(x,y) \in \alpha \Omega_m$ and hence $g^M(x,y) \in \alpha \Omega_m$
because $g$ is linearly homogeneous, hence $\alpha \Omega_m$ is forward invariant under $g$.
We have therefore shown that $g^{k+M}(\Omega_m) \subset \alpha \Omega_m$.
Therefore for all $j \ge 0$ we have $g^{j(k+M)}(\Omega_m) \subset \alpha^j \Omega_m \to (0,0)$ as $j \to \infty$.
There exists $r > 0$ such that $g^M(B_r) \subset \Delta_0 \subset \Omega_m$,
thus for all $(x,y) \in B_r$ we have $g^n(x,y) \to (0,0)$ as $n \to \infty$.
Thus $\b0$ is an asymptotically stable fixed point of $g$ by Lemma \ref{le:clh}(i).
\end{proof}


\begin{thebibliography}{10}
{\footnotesize

\bibitem{Ab00}
M.~Abadie.
\newblock Dynamical simulation of rigid bodies: Modelling of frictional
  contact.
\newblock In B.~Brogliato, editor, {\em Impacts in Mechanical Systems: Analysis
  and Modelling.} Springer-Verlag, New York, 2000.

\bibitem{AlYo92}
J.C. Alexander, J.A. Yorke, Z.~You, and I.~Kan.
\newblock Riddled basins.
\newblock {\em Int. J. Bifurcation Chaos}, 2(4):795--813, 1992.

\bibitem{An03}
V.~Andreasen.
\newblock Dynamics of annual influenza {A} epidemics with immuno-selection.
\newblock {\em J. Math. Biol.}, 46:504--536, 2003.

\bibitem{AwLa03}
J.~Awrejcewicz and C.~Lamarque.
\newblock {\em Bifurcation and Chaos in Nonsmooth Mechanical Systems.}
\newblock World Scientific, Singapore, 2003.

\bibitem{BaYo98}
S.~Banerjee, J.A. Yorke, and C.~Grebogi.
\newblock Robust chaos.
\newblock {\em Phys. Rev. Lett.}, 80(14):3049--3052, 1998.

\bibitem{BlCz99}
B.~Blazejczyk-Okolewska, K.~Czolczynski, T.~Kapitaniak, and J.~Wojewoda.
\newblock {\em Chaotic Mechanics in Systems with Impacts and Friction}.
\newblock World Scientific, Singapore, 1999.

\bibitem{BoCo04}
P.~Bolzern, P.~Colaneri, and G.~De~Nicolao.
\newblock On almost sure stability of discrete-time {M}arkov jump linear
  systems.
\newblock In {\em Proceedings of the 43rd IEEE Conference on Decision and
  Control.}, pages 3204--3208, 2004.

\bibitem{Br99}
B.~Brogliato.
\newblock {\em Nonsmooth Mechanics: Models, Dynamics and Control.}
\newblock Springer-Verlag, New York, 1999.

\bibitem{DeGr76}
M.~Denker, C.~Grillenberger, and K.~Sigmund.
\newblock {\em Ergodic Theory on Compact Spaces.}, volume 527 of {\em Lecture
  Notes in Mathematics.}
\newblock Springer-Verlag, New York, 1976.

\bibitem{Di03}
M.~di~Bernardo.
\newblock Normal forms of border collision in high dimensional non-smooth maps.
\newblock In {\em Proceedings IEEE ISCAS, Bangkok, Thailand}, volume~3, pages
  76--79, 2003.

\bibitem{DiBu01}
M.~di~Bernardo, C.J. Budd, and A.R. Champneys.
\newblock Normal form maps for grazing bifurcations in $n$-dimensional
  piecewise-smooth dynamical systems.
\newblock {\em Phys. D}, 160:222--254, 2001.

\bibitem{DiBu08}
M.~di~Bernardo, C.J. Budd, A.R. Champneys, and P.~Kowalczyk.
\newblock {\em Piecewise-smooth Dynamical Systems. Theory and Applications.}
\newblock Springer-Verlag, New York, 2008.

\bibitem{DiNo08}
M.~di~Bernardo, A.~Nordmark, and G.~Olivar.
\newblock Discontinuity-induced bifurcations of equilibria in piecewise-smooth
  and impacting dynamical systems.
\newblock {\em Phys. D}, 237:119--136, 2008.

\bibitem{DoBa06}
Y.~Do and H.K. Baek.
\newblock Dangerous border-collision bifurcations of a piecewise-smooth map.
\newblock {\em Comm. Pure Appl. Anal.}, 5(3):493--503, 2006.

\bibitem{DoKi08}
Y.~Do, S.D. Kim, and P.S. Kim.
\newblock Stability of fixed points placed on the border in the piecewise
  linear systems.
\newblock {\em Chaos Solitons Fractals}, 38(2):391--399, 2008.

\bibitem{Fe78}
M.I. Feigin.
\newblock On the structure of {$C$}-bifurcation boundaries of
  piecewise-continuous systems.
\newblock {\em J. Appl. Math. Mech.}, 42(5):885--895, 1978.
\newblock Translation of {\em Prikl.~Mat.~Mekh.}, 42(5):820-829, 1978.

\bibitem{GaBa05}
A.~Ganguli and S.~Banerjee.
\newblock Dangerous bifurcation at border collision: When does it occur?
\newblock {\em Phys. Rev. E.}, 71(5):057202, 2005.

\bibitem{Ga92}
L.~Gardini.
\newblock Some global bifurcations of two-dimensional endomorphisms by use of
  critical lines.
\newblock {\em Nonlinear Anal.}, 18(4):361--399, 1992.

\bibitem{Gl01}
P.~Glendinning.
\newblock Milnor attractors and topological attractors of a piecewise linear
  map.
\newblock {\em Nonlinearity}, 14(2):239--257, 2001.

\bibitem{Gl15b}
P.~Glendinning.
\newblock Bifurcation from stable fixed point to ${N}$-dimensional attractor in
  the border collision normal form.
\newblock {\em Nonlinearity}, 28:3457--3464, 2015.

\bibitem{Gl16e}
P.~Glendinning.
\newblock Bifurcation from stable fixed point to {2D} attractor in the border
  collision normal form.
\newblock {\em To appear:} IMA J.~Appl.~Math., 2016.

\bibitem{GlWo09}
P.~Glendinning and C.H. Wong.
\newblock Border collision bifurcations, snap-back repellers, and chaos.
\newblock {\em Phys. Rev. E}, 79:025202, 2009.

\bibitem{GlWo11}
P.~Glendinning and C.H. Wong.
\newblock Two dimensional attractors in the border collision normal form.
\newblock {\em Nonlinearity}, 24:995--1010, 2011.

\bibitem{HaCh08}
W.M. Haddad and V.~Chellaboina.
\newblock {\em Nonlinear Dynamical Systems and Control: {A} {L}yapunov-Based
  Approach.}
\newblock Princeton University Press, Princeton, NJ, 2008.

\bibitem{HaAb04}
M.A. Hassouneh, E.H. Abed, and H.E. Nusse.
\newblock Robust dangerous border-collision bifurcations in piecewise smooth
  systems.
\newblock {\em Phys. Rev. Lett.}, 92:070201, 2004.

\bibitem{Jo03}
M.~Johansson.
\newblock {\em Piecewise Linear Control Systems.}, volume 284 of {\em Lecture
  Notes in Control and Information Sciences.}
\newblock Springer-Verlag, New York, 2003.

\bibitem{KaMc10}
S.~Kalikow and R.~McCutcheon.
\newblock {\em An Outline of Ergodic Theory.}, volume 122 of {\em Cambridge
  studies in advanced mathematics.}
\newblock Cambridge University Press, Cambridge, 2010.

\bibitem{KaMa99}
T.~Kapitaniak and Yu. Maistrenko.
\newblock Riddling bifurcations in coupled piecewise linear maps.
\newblock {\em Phys. D}, 126:18--26, 1999.

\bibitem{LaTr02}
V.~Lakshmikantham and D.~Trigiante.
\newblock {\em Theory of Difference Equations: {N}umerical Methods and
  Applications.}
\newblock Marcel Dekker, Inc., New York, 2002.

\bibitem{LiAn09}
H.~Lin and P.J. Antsaklis.
\newblock Stability and stabilization of switched linear systems: {A} survey of
  recent results.
\newblock {\em IEEE. Trans. Auto. Contr.}, 54(2):308--322, 2009.

\bibitem{MaMc76}
J.E. Marsden and M.~McCracken.
\newblock {\em The {H}opf Bifurcation and its Applications}.
\newblock Springer-Verlag, New York, 1976.

\bibitem{Mi85}
J.~Milnor.
\newblock On the concept of attractor.
\newblock {\em Commun. Math. Phys.}, 99:177--195, 1985.

\bibitem{MiGa96}
C.~Mira, L.~Gardini, A.~Barugola, and J.~Cathala.
\newblock {\em Chaotic Dynamics in Two-Dimensional Noninvertible Maps.},
  volume~20 of {\em Nonlinear Science}.
\newblock World Scientific, Singapore, 1996.

\bibitem{NuYo92}
H.E. Nusse and J.A. Yorke.
\newblock Border-collision bifurcations including ``period two to period
  three'' for piecewise smooth systems.
\newblock {\em Phys. D}, 57:39--57, 1992.

\bibitem{OeHi96}
M.~Oestreich, N.~Hinrichs, and K.~Popp.
\newblock Bifurcation and stability analysis for a non-smooth friction
  oscillator.
\newblock {\em Arch. Appl. Mech.}, 66:301--314, 1996.

\bibitem{PuSu06}
T.~Puu and I.~Sushko, editors.
\newblock {\em Business Cycle Dynamics: Models and Tools.}
\newblock Springer-Verlag, New York, 2006.

\bibitem{RoFi10}
H.L. Royden and P.M. Fitzpatrick.
\newblock {\em Real Analysis.}
\newblock Prentice Hall, New York, 2010.

\bibitem{Si14d}
D.J.W. Simpson.
\newblock On the relative coexistence of fixed points and period-two solutions
  near border-collision bifurcations.
\newblock {\em Appl. Math. Lett.}, 38:162--167, 2014.

\bibitem{Si16}
D.J.W. Simpson.
\newblock Border-collision bifurcations in $\mathbb{R}^n$.
\newblock {\em SIAM Rev.}, 58(2):177--226, 2016.

\bibitem{So03}
H.H. Sohrab.
\newblock {\em Basic Real Analysis.}
\newblock Birkh\"{a}user, Boston, 2003.

\bibitem{SzOs09}
R.~Szalai and H.M. Osinga.
\newblock Arnol'd tongues arising from a grazing-sliding bifurcation.
\newblock {\em SIAM J. Appl. Dyn. Sys.}, 8(4):1434--1461, 2009.

\bibitem{TaLa12}
S.-C. Tan, Y.-M. Lai, and C.K. Tse.
\newblock {\em Sliding Mode Control of Switching Power Converters.}
\newblock CRC Press, Boca Raton, FL, 2012.

\bibitem{Ti02}
P.H.E. Tiesinga.
\newblock Precision and reliability of periodically and quasiperiodically
  driven integrate-and-fire neurons.
\newblock {\em Phys. Rev. E}, 65(4):041913, 2002.

\bibitem{Ut92}
V.I. Utkin.
\newblock {\em Sliding Modes in Control Optimization.}
\newblock Springer-Verlag, New York, 1992.

\bibitem{Wa82}
P.~Walters.
\newblock {\em An Introduction to Ergodic Theory.}
\newblock Springer-Verlag, New York, 1982.

\bibitem{WiDe00}
M.~Wiercigroch and B.~De~Kraker, editors.
\newblock {\em Applied Nonlinear Dynamics and Chaos of Mechanical Systems with
  Discontinuities.}, Singapore, 2000. World Scientific.

\bibitem{ZhMo06b}
Z.T. Zhusubaliyev, E.~Mosekilde, S.~Maity, S.~Mohanan, and S.~Banerjee.
\newblock Border collision route to quasiperiodicity: Numerical investigation
  and experimental confirmation.
\newblock {\em Chaos}, 16(2):023122, 2006.

}
\end{thebibliography}
\end{document}